\documentclass[10pt]{article}
\usepackage{amssymb}
\usepackage{amsfonts}
\usepackage{graphicx}
\usepackage{epstopdf}
\usepackage[center]{caption}
\usepackage{color}
\usepackage{amsmath}
\usepackage{enumerate}

\newtheorem{theorem}{Theorem}[section]

\newtheorem{assumption}[theorem]{Assumption}

\newtheorem{claim}[theorem]{Claim}

\newtheorem{corollary}[theorem]{Corollary}

\newtheorem{definition}[theorem]{Definition}

\newtheorem{lemma}[theorem]{Lemma}

\newtheorem{proposition}[theorem]{Proposition}
\newtheorem{remark}[theorem]{Remark}

\newenvironment{proof}[1][Proof]{\noindent\textit{#1.} }{\hfill \rule{0.5em}{0.5em}}

\newcommand{\R}{\mathbb{R}}

\begin{document}

\title{\textbf{Existence and uniqueness of propagating terraces}}

\author{\textsc{Thomas Giletti$^{a}$\footnote{This work was initiated when the first author was visiting the University of Tokyo with the support of the Japanese Society for the Promotion of Science. The first author also acknowledges support from the NONLOCAL project (ANR-14-CE25-0013) funded by the French National Research Agency (ANR).} and Hiroshi Matano$^b$} \\
$^{a}${\small Univ. Lorraine, Institut Elie Cartan Lorraine, UMR 7502, Vandoeuvre-l\`{e}s-Nancy, F-54506, France}\\
$^b${\small Graduate School of Mathematical Sciences, University of Tokyo, Komaba, Tokyo 153-8914, Japan}
}

\maketitle

\begin{abstract}
This work focuses on dynamics arising from reaction-diffusion equations, where the profile of propagation is no longer characterized by a single front, but by a layer of several fronts which we call a propagating terrace. This means, intuitively, that transition from one equilibrium to another may occur in several steps, that is, successive phases between some intermediate stationary states. We establish a number of properties on such propagating terraces in a one-dimensional periodic environment, under very wide and generic conditions. We are especially concerned with their existence, uniqueness, and their spatial structure. Our goal is to provide insight into the intricate dynamics arising from multistable nonlinearities.
\end{abstract}

\section{Introduction}\label{sec:intro}

We consider the following reaction-diffusion equation in one space dimension:
\begin{equation}\tag{$E$}\label{eqn1}
\partial_t u(t,x)= \partial_{x} (a(x) \partial_x  u(t,x))+ f(x,u(t,x)), \ \forall (t,x) \in \R \times \R,
\end{equation}
where, throughout the paper, the functions $a$ and $f$ satisfy the following regularity and periodicity assumptions:
\begin{equation}\label{eqn:a}
0< a \in C^2 (\R,\R) \mbox{ and } a (x+L) \equiv a(x),
\end{equation}
\begin{equation}\label{eqn:f}
f \in C^1 (\R^2,\R) \mbox{ and } f(x+L,u) \equiv f(x,u),
\end{equation}
for some $L>0$.\\

Furthermore, we will always assume that there exist at least two equilibrium states, namely 0 and $p$. More precisely,
\begin{equation}\label{eqn:f_0}
f(x,0) \equiv 0,
\end{equation}
and there exists an $L$-periodic and positive stationary solution:
\begin{equation}\label{eqn:f_p}
\left\{
\begin{array}{l}
\partial_{x} (a(x) \partial_x  p) + f(x,p) =0, \ \forall x \in \R , \vspace{3pt}\\
p ( x+L) \equiv p (x), \ p(x)>0.
\end{array}
\right.
\end{equation}
Clearly, the function $p$ is also a stationary solution of the following equation, $L$-periodic counterpart of (\ref{eqn1}):
\begin{equation}\tag{$E_{per}$}\label{eqn1-per}
\left\{
\begin{array}{l}
\partial_t u(t,x)= \partial_{x} (a(x) \partial_x  u(t,x)) + f(x,u(t,x)), \ \forall (t,x) \in \R \times \R, \vspace{3pt}\\
u (t, x+L) \equiv u (t,x), \ \forall (t,x) \in \R \times \R.
\end{array}
\right.
\end{equation}
It is obvious that any solution of (\ref{eqn1-per}) is also a solution of (\ref{eqn1}).\\

The goal of this paper is to investigate propagation dynamics between the two equilibria 0 and $p$, for very large classes of nonlinearities, including but not limited to the classical monostable, ignition or bistable cases. The present work is a continuation of our  previous paper~\cite{DGM}, where we studied propagation from~0 to $p$ under an additional stability assumption on $p$. It also generalizes the earlier work of Fife and Mc Leod~\cite{FifMcL77}, in which some related problems were studied for spatially homogeneous multistable nonlinearities.

Our analysis will highlight complex dynamics which, unlike in the standard cases, cannot be described by a single front. They will instead involve finite sequences of fronts, that we will call propagating terraces. We set forth this notion as a natural generalization of traveling waves, providing a new and robust framework for describing dynamics of a highly general class of periodic reaction-diffusion equations.

\subsection{Propagating terraces: some definitions}

The key concept throughout this paper is that of a propagating terrace, which we define below. It is a generalization of the classical notion of a traveling wave.

Let us first recall the notion of pulsating traveling waves~\cite{BH02,Wein02}. A pulsating traveling wave solution of \eqref{eqn1} connecting two periodic stationary solutions $p_1$ and $p_2$ of \eqref{eqn1-per} with speed $c$ is a particular entire solution of the type $U(x-ct,x)$ where $U(z,x)$ is periodic in its second variable, and satisfies the asymptotics $U (+\infty,\cdot)=p_2 (\cdot)$ and $U(-\infty,\cdot) = p_1 (\cdot)$. Moreover, when the speed $c=0$, we say that it is a stationary wave.

\begin{definition}\label{def:terrace}
A \textbf{propagating terrace} $\mathcal{T}$ connecting 0 to $p$ is a couple of two finite sequences $(p_k)_{0 \leq k \leq N}$ and $(U_k)_{1 \leq k \leq N}$ such that:
\begin{itemize}
\item The $p_k$ are $L$-periodic stationary solutions of $($\ref{eqn1}$)$ satisfying
$$p=p_0 >p_1 >...>p_N =0.$$
\item For any $1 \leq k \leq N$, $U_k$ is a pulsating traveling wave solution of $($\ref{eqn1}$)$ connecting $p_k$ to $p_{k-1}$ with speed $c_k \in \R$.
\item The sequence $(c_k)_k$ satisfies $$c_1 \leq c_2 \leq ... \leq c_N .$$
\end{itemize}
Moreover, we will refer to the sequence $(p_k)_{1\leq k \leq N}$ as the \textbf{platforms} of the terrace~$\mathcal{T}$, and hence to $N$ as the number of platforms of $\mathcal{T}$.
\end{definition}
Similarly, for any pair of $L$-periodic stationary solutions $p$, $q$ satisfying $p  > q$, we can define a propagating terrace connecting $q$ to $p$ in a completely analogous manner.\\

Propagating terraces were already introduced under this name in our previous work~\cite{DGM}. The definition we use here is slightly more general because the speeds~$c_k$ are no longer required to be nonnegative and, therefore, the upper and lower parts of the terrace may spread in opposite directions. This is natural in some situations, such as when 0 and $p$ are both unstable (for instance, in the so-called heterozygote superior case from genetics~\cite{AW75}). Incidentally, as mentioned in our previous paper~\cite{DGM}, the notion of terrace already appears in the work of Fife and Mc Leod~\cite{FifMcL77} under the name ``minimal decomposition". However, their analysis was restricted to homogeneous multistable nonlinearities, for which phase plane methods suffice to determine the structure of the terrace. Our results will extend theirs to more general classes of nonlinearities with spatial periodicity, for which the usual ODE tools no longer work. We also refer to~\cite{Risler1,Risler2} for related results in the case of homogeneous systems with a gradient structure, and to~\cite{Polacik} for convergence results in higher dimension by symmetrisation techniques.

In the above definition, ordering the speeds of terraces is fundamental for their purpose. Indeed, by analogy with single traveling waves in the standard cases~\cite{Bramson83,FifMcL77,Giletti}, they are to describe dynamics arising in the large time behavior of solutions of the Cauchy problem associated with equation \eqref{eqn1}. For instance, convergence to a terrace from Heaviside type initial data $H(a-x) p(x)$, where $a \in \R$ and $H$ is the classical Heaviside function, was already proved in~\cite{DGM} under the additional assumption that the solution converges locally uniformly to~$p$ for some compactly supported initial datum. Such solutions are close to~$0$ (respectively $p$) on the far right (respectively left) of the domain for any positive time, so that lower level sets must spread faster to the right than the upper ones. In other words, a terrace should be of an intuitively appropriate shape, as illustrated in Figure~\ref{fig:propag_terrace}, at least as $t \rightarrow +\infty$.
\begin{figure}[h]
\centering
\includegraphics[width=1\textwidth]{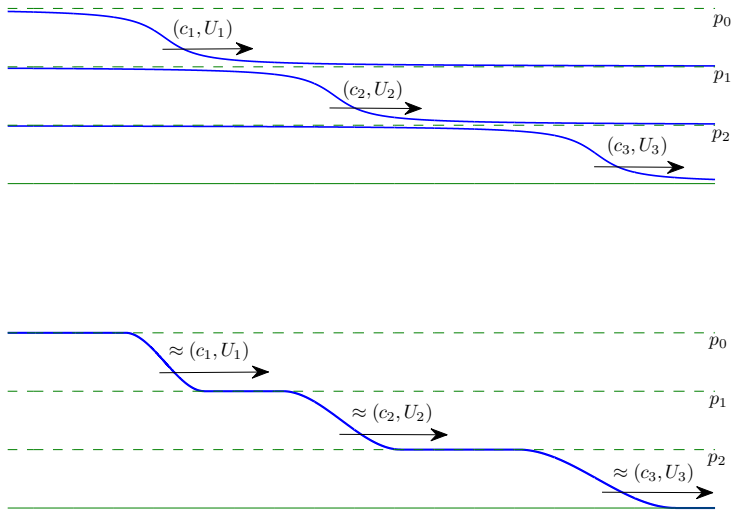}
\caption{(above) A three platforms terrace;\\ (below) A terrace-shaped profile of propagation}
\label{fig:propag_terrace}
\end{figure}
\begin{remark}
We point out to the reader that, for convenience, we always say that a wave connects $q_2$ to $q_1$ if $q_2$ is the limiting state on the right end of the spatial domain, and $q_1$ on the left. However, when the speed is negative, the front actually goes to $q_2$ in large positive time.
\end{remark}

\subsubsection*{Some further definitions}

Some particular terraces distinguish themselves and seem to play a more important role. This is to be expected from standard cases. Indeed, for a classical monostable KPP nonlinearity, it is known that there exists a continuum of admissible speeds $[c,+\infty)$ for traveling waves. However, only the wave with minimal speed attracts large classes of solutions, such as those with compactly supported initial data (see~\cite{Bramson83,Lau85,Uchi78} in the simpler homogeneous case, and more recently~\cite{Giletti,HNRR} for the general periodic framework). 

Here we note that the wave with minimal speed can also be characterized as the ``steepest" one among all the traveling waves connecting the same stationary states, in the sense to be defined in Definition~\ref{def:steep} below. The link between speed and steepness has been known in many typical cases, such as in the homogeneous framework~(see for instance the seminal paper of Kolmogorov, Petrovsky and Piskunov~\cite{KPP}, as well as~\cite{FifMcL77} for more general nonlinearities), and will be proved here in Section~\ref{sec:minspeed} for large classes of spatially periodic nonlinearities.

Therefore, it is also natural to classify traveling waves (or more generally, in the present framework, terraces) according to their steepness, and to expect the steepest one to play the most important role in the dynamics. This was already fundamental in our previous work~\cite{DGM}, and will again be in the present work.

\begin{definition}\label{def:steep}
Let two functions $u_1 (t,x)$ and $u_2 (t,x)$. We say that $u_1$ is \textbf{steeper than} $u_2$ if for any $x_1$, $t_1$ and $t_2$  in $\R$,
$$u_1 (t_1, x_1 ) = u_2 (t_2,x_1) \Longrightarrow 
\left\{
\begin{array}{l}
u_1 (t_1,x) \geq u_2 (t_2,x) \mbox{ for any } x \leq x_1 , \vspace{3pt} \\
 u_1 (t_1,x) \leq u_2 (t_2,x) \mbox{ for any } x \geq x_1.
\end{array}
\right.
$$
\end{definition}

Here we remark that $u_1$ and $u_2$ are compared at two arbitrary time moments~$t_1$ and~$t_2$ not necessarily equal. Therefore, according to this definition, steepness is to be understood here up to any time shift. The definition is otherwise intuitive, as it simply states that a function is steeper than another if they (again, up to any time shift) can only intersect once, and that the steepest function then has to be below the other on the right and above on the left. Of course, it is trivial to extend this definition to functions that depend on $x$ only, seeing them as constant with respect to time. The reader should also note that according to this definition, if the images of $u_1$ and $u_2$ are disjoint, then they are steeper than each other.\\

We now introduce, as announced above, some very steep particular propagating terraces:
\begin{definition}\label{def:terrace}
A propagating terrace $\mathcal{T} = \left( (p_k)_{0 \leq k \leq N},(U_k)_{1 \leq k \leq N} \right)$ is said to be \textbf{semi-minimal} if, for each $k$, $U_k$ is steeper than any other traveling wave that connects $p_{k-1}$ to $p_k$.

The terrace $\mathcal{T}$ is called \textbf{minimal} if it is semi-minimal and satisfies $$\{ p_k \} \subset \{q_k \}$$ for any other propagating terrace $( (q_k)_k , (V_k)_k )$ that connects 0 to~$p$.
\end{definition}

The fact that the set $\{p_k\}$ of the platforms of a minimal terrace is included in any set of platforms of any other terrace can indeed be seen as a steepness property, as it implies that any $p_k$ is steeper than any part of any terrace. We will in fact prove that such terraces have even stronger steepness properties, such as being steeper than any other entire solution of~\eqref{eqn1}. 

To convince oneself of the strength of such properties, one could already easily check that the minimal terrace, if it exists and if its speeds are all non zero, is unique up to time shifts. More precisely, all minimal terraces share the same platforms and, between each consecutive platforms, consist of identically equal up to time shift traveling waves. This is Theorem~\ref{th:unique_min} below.\\

Let us now give one last definition, which we draw from~\cite{FifMcL77}. Unlike the two above, it is actually slightly weaker than our notion of propagating terrace.
\begin{definition}\label{def:decomp}
A \textbf{decomposition} between 0 and $p$ is a finite sequence of $L$-periodic stationary solutions $(p_k)_{0 \leq k \leq N}$ such that $p_0 = p > p_1 > ... > p_N = 0$, and for any $1 \leq k \leq N$, there exists a pulsating traveling wave $U_k$ connecting $p_k$ to $p_{k-1}$.

We say that 0 and $p$ are \textbf{connected} if there exists a single traveling wave connecting~0 to~$p$.
\end{definition}

This is very similar to the definition of platforms of a propagating terrace, except that it does not require the speeds of the traveling waves to be ordered anymore. This means as explained above that a decomposition, unless of course it is also a terrace, does not provide a rightful candidate to look at the large time profile of solutions of \eqref{eqn1}. However, on the other hand, existence of a decomposition is much easier to check: for instance, it exists if~$f$ is homogeneous and has a finite number of zeros as a function of~$u$, or more generally, if~$f$ can be rewritten as some concatenation of standard nonlinearities. We will show with Theorem~\ref{th:exists} that existence of a decomposition is actually enough to ensure existence of terraces.

\subsection{Some assumptions}\label{sec:assumptions}

We have already highlighted above the particular attention we will pay to the steepness of traveling waves and propagating terraces, and to how it relates to their speeds. We are only able to establish such a relation, which we will detail in Section~\ref{sec:minspeed} (see in particular Theorem~\ref{th:lem_speeds}), under additional assumptions. Let us briefly enounce them below.\\

First, we will denote, for any $L$-periodic function $g$, by $\mu_g$ the principal eigenvalue of the following problem:
\begin{equation*}
\left\{
\begin{array}{l}
\displaystyle \partial_{x} (a \,\partial_x \phi ) + g \phi  =  \mu_g \phi \ \mbox{ in } \R, \vspace{5pt}\\
\phi >0 \mbox{ and  $L$-periodic}.
\end{array}
\right.
\end{equation*}
When $q$ is a stationary solution of (\ref{eqn1-per}), it is commonly said to be linearly stable (respectively unstable) when $\mu_g <0$ (respectively $\mu_g >0$) with the function $g(x) = \partial_u f(x,q (x))$.

Occasionally, we will be led to consider the following two assumptions on the stability of equilibria:
\begin{assumption}\label{assumption-speed1}
For any $L$-periodic positive stationary solution $q$, with $0 \leq q \leq p$, that is stable from below with respect to~$($\ref{eqn1-per}$)$, there exist $\delta >0$ and an $L$-periodic function $g$ such that $\mu_g \leq 0$ and
$$\partial_u f (x,u) \leq g (x) \; \mbox{ for all } x \in  \R \mbox{ and } u \in \left[ q(x)-\delta, q(x)\right].$$
For any $L$-periodic positive stationary solution $q$, with $0 \leq q \leq p$, that is stable from above with respect to~$($\ref{eqn1-per}$)$, there exist $\delta >0$ and an $L$-periodic function $g$ such that $\mu_g \leq 0$ and
$$\partial_u f (x,u) \leq g (x)\; \mbox{ for all } x \in  \R \mbox{ and } u \in \left[ q(x), q(x)+\delta\right].$$
\end{assumption}

\begin{assumption}\label{assumption-speed2}
For any $L$-periodic positive stationary solution $q$, with $0 < q < p$, that is stable with respect to~$($\ref{eqn1-per}$)$ from either below or above, there exist $\delta >0$ and an $L$-periodic function $g$ such that $\mu_g \leq 0$ and
$$\partial_u f (x,u) \leq g (x)\; \mbox{ for all } x \in  \R \mbox{ and } u \in \left[ q(x)-\delta, q(x)+\delta \right].$$
Furthermore, there exist $\delta >0$ and an $L$-periodic function $g$ such that $\mu_g \leq 0$ and
$$\partial_u f (x,u) \leq g (x)\; \mbox{ for all } x \in  \R \mbox{ and } u \in \left[ 0, \delta \right] \cup \left[p(x)-\delta,p(x) \right].$$
\end{assumption}
Let us remark that, when \eqref{eqn1-per} does not admit any degenerate equilibrium, then Assumption~\ref{assumption-speed1} is satisfied, as well as Assumption~\ref{assumption-speed2} provided that 0 and~$p$ are linearly stable. Indeed, for any linearly stable stationary solution $q$ of \eqref{eqn1-per}, one can choose $\varepsilon$ small enough so that $\mu_g \leq 0$ with $g = \partial_u f (x,q) + \varepsilon$, and $\partial_u f (x,u) \leq g(x)$ in some set $(x,u) \in \R \times \left[q(x)-\delta,q(x)+\delta\right]$. By degenerate equilibrium, we mean here that it is neither linearly stable nor linearly unstable.
%
%

\subsection{Main results}

Before announcing our results, we remind the reader that the only standing hypotheses are the regularity and periodicity of $a$ and $f$ \eqref{eqn:a}-\eqref{eqn:f}, and the existence of two periodic stationary states \eqref{eqn:f_0} and \eqref{eqn:f_p}. Whenever we make any of the additional Assumptions~\ref{assumption-speed1} or~\ref{assumption-speed2}, we will state it explicitly.

We now proceed to the statements of our main results, some of which we have already evoked above. As announced, the purpose of our notion of propagating terraces is to give relevant insight on the dynamics of equation~\eqref{eqn1} for complex multistable nonlinearities. This approach was already justified in~\cite{DGM} where convergence to a minimal terrace was shown from an Heaviside type initial datum, in the case where it spreads to the right with positive speed.

Among our results, we will give a new and more general convergence theorem, that further strenghtens the reliability and robustness of the notion of terraces. However, in spite of its relevance, this notion needs to be better understood in view of its application. One could for instance wonder, for a particular~$f$, what we can foreknow on the shape of terraces, and in particular whether we can predict which platforms it should contain. This challenge will be addressed through looking at important properties of terraces, in particular at their uniqueness and steepness. We will illustrate, later on this paper, how those universal results do provide the wanted tools, and how they can be applied to particular examples.\\

We begin with an existence result for the propagating terrace. Apart from the standing hypotheses \eqref{eqn:a}, \eqref{eqn:f}, \eqref{eqn:f_0} and \eqref{eqn:f_p}, this result does not require any of the additional assumptions mentioned in subsection~\ref{sec:assumptions}: 
\begin{theorem}[From decompositions to terraces]\label{th:exists}
The following three properties are equivalent:
\begin{enumerate}[$(1)$]
\item There exists a decomposition between 0 and $p$.
\item There exists a propagating terrace connecting 0 to $p$.
\item There exists a minimal propagating terrace $((p_k)_k,(U_k)_k)$ connecting 0 to~$p$; moreover, any $p_k$ and $U_k$ is steeper than any other entire solution that lies between 0 and~$p$.
\end{enumerate}
\end{theorem}

Note that it is immediate, from our definitions, that $(3) \Rightarrow (2) \Rightarrow (1)$. Therefore, the main point of this theorem is that existence of any decomposition is enough to guarantee existence of a minimal terrace, and that every component of this minimal terrace is steeper than any entire solution between 0 and~$p$. Here, we highlight again the importance of such steepness. Indeed, the fact that any platform~$p_k$ of a minimal terrace is steeper than any other entire solution implies that~$p_k$ is included in any decomposition connecting 0 and~$p$. In other words, it is an immediate corollary of the above theorem that, from any decomposition connecting 0 and~$p$, one can extract the platforms of minimal terraces:
\begin{corollary}\label{th:exists_cor}
If  $((p_k),(U_k))$  is a minimal terrace connecting 0  to $p$, 
then for any decomposition  $(q_k)$ between 0 and $p$, one has 
$$\{ p_k \} \subset \{q_k\}.$$
\end{corollary}

Note that although the minimal terrace may not be unique, it is trivial that all minimal terraces share the same platforms, so that this corollary indeed easily follows from the above theorem.\\

In our previous work~\cite{DGM}, we only treated terraces that are moving to the right, that is, with $c_k >0$ for all $k$. Indeed we assumed the existence of a solution with compactly supported initial datum that converges to~$p$ from below as $t \to +\infty$, locally uniformly on $\mathbb{R}$. In such a situation, it is clear that $c_k >0$ for every~$k$. In the present paper, we do not make such an assumption. As a result, we allow $c_k <0$ for some integers~$k$. Thus, our Theorem~\ref{th:exists} completely covers the earlier existence result of Fife and McLeod~\cite{FifMcL77} for the homogeneous case, and generalizes their result to spatially periodic equations. Our result also covers such cases as multi-layered monostable nonlinearities that were not treated in~\cite{FifMcL77}.

It should be noted that there does not always exist a decomposition between~0 and $p$. A simple counter example is the case where $f \equiv 0$. A slightly less trivial example is a reversed combustion nonlinearity, that is $f(u) =0$ ($u=0$, $\theta \leq u \leq 1$) and $f(u) >0$ ($0< u < \theta$). On the other hand, apart from these highly degenerate nonlinearities, all the nonlinearities that we find in many standard physical models possess a decomposition. 
The following proposition, which is not entirely new, gives a simple sufficient condition for the existence of a decomposition.
\begin{proposition}\label{proposition}
Assume that there exist a finite number of stationary solutions $p =: q_0 > q_1 > ... > q_{m-1} > q_m := 0$ of \eqref{eqn1-per} such that there exists no stationary solution strictly between $q_{j-1}$ and $q_j$ for any $j=1,...,m$.
Then there exists a decomposition between 0 and $p$, hence a minimal terrace.
\end{proposition}
\begin{proof}
Since there exists no other stationary solution of \eqref{eqn1-per} between $q_{j-1}$ and $q_j$, we see from \cite[Corollary~4.5]{Matano79} that either all solutions whose initial data lie between $q_{j-1}$ and $q_j$ converge to $q_{j-1}$ as $t \to +\infty$ or they converge to $q_j$ as $t \to +\infty$. Corollary~4.5 of~\cite{Matano79} is for problems under the Dirichlet, Neumann or Robin boundary conditions but precisely the same argument carries over to the periodic boundary conditions. Thus by the result of Weinberger~\cite{Wein02}, there exists a pulsating traveling wave solution of~\eqref{eqn1} connecting $q_j$ to $q_{j-1}$ or vice versa, $j=1,...,m$. These traveling waves constitute a decomposition between $0$ and $p$. Hence a minimal terrace exists by Theorem~\ref{th:exists}.
\end{proof}\\

In the spatially homogeneous case where $a$ and $f$ are independent of $x$, the assumption in Proposition~\ref{proposition} can be restated that $f(u)$ has finitely many zeroes between 0 and $p$. We remark that this finiteness assumption is by no means necessary for the existence of a decomposition. A classical example is a combustion nonlinearity, that is, $f(u) = 0$ ($0 \leq u \leq \theta$, $u =1$) and $f (u) >0$ ($\theta < u < 1$), for which a traveling wave connecting 1 to 0 is known to exist~\cite{AW75}. A more general sufficient condition for the existence of a decomposition in the spatially homogeneous case is found in \cite[Theorem~1.2]{Polacik2}.

In the spatially periodic case, our earlier paper~\cite[Theorem~1.10]{DGM} gives another sufficient condition, namely the existence of a compactly supported initial datum from which the solution of \eqref{eqn1} converges to $p$ from below as $t \to +\infty$; this condition admits, for example, a spatially periodic combustion type nonlinearity, which does not satisfy the assumption of Proposition~\ref{proposition}. Nonetheless Proposition~\ref{proposition} covers many important examples and its assumption is simple and relatively easy to verify.

The existence of a pulsating traveling wave for a spatially periodic bistable nonlinearity is a particular case to which Proposition~\ref{proposition} applies. This example will be treated in Section~\ref{sec:exemples} along with other examples.\\

The main idea in the proof of Theorem~\ref{th:exists} is first to show that the solution starting from a Heaviside function like initial datum becomes less and less steep as time passes, while remaining steeper than any entire solution that lies between 0 and~$p$, including traveling waves. This idea is basically along the same lines as in our previous paper~\cite{DGM} and is somewhat inspired by the work of Kolmogorov, Petrovsky and Piskunov~\cite{KPP}. By slightly modifying the argument, we obtain the following theorem on the long time behavior of more general solutions:
\begin{theorem}[Convergence to a terrace]\label{th:CV}
Assume that there exists a minimal terrace $\mathcal{T} = ((p_k)_k,(U_k)_k)$ such that $c_k \neq 0$ for any $k$, and let~$u_0$ be a piecewise continuous function such that
$$u_0 \mbox{ is steeper than any $p_k$ and $U_k$},$$
$$u(x)-p(x) \rightarrow 0 \mbox{ as } x\rightarrow -\infty \mbox{, and } u(x) \rightarrow 0 \mbox{ as } x \rightarrow +\infty.$$
Then the associated solution $u$ of the Cauchy problem for \eqref{eqn1} converges to~$\mathcal{T}$ in the following sense:\begin{enumerate}[$(i)$]
\item There exist functions $(m_k (t))_{1 \leq k \leq N}$ with
$m_k (t) = o(t)$ as $t \rightarrow +\infty$ such that
\begin{equation}\label{eqn:CVcor1}
\begin{split}
u(t,x +c_k (t-m_k (t)) )- U_k (t-m_k (t),x +
c_k (t-m_k (t))) \quad \\
\to 0 \ \ \ \hbox{as} \ \ t\to +\infty,
\end{split}
\end{equation}
locally uniformly on $\R$.
\item For any $1 \leq k \leq N-1$, we have that
$$ c_{k+1} ( t - m_{k+1} (t)) - c_k (t - m_k (t)) \to +\infty, \ \mbox{ as $ t \to +\infty$}.$$
Moreover, for any $\delta >0$, there exist $C>0$ and $T>0$ such that, for any $t\geq T$ and
$1 \leq k \leq N-1$:
$$
\| u(t,\cdot  )- p_k (\cdot) \|_{L^\infty ([c_k
(t-m_k (t))  + C , c_{k+1} (t-m_{k+1} (t)) - C])} \leq \delta,
$$
together with
$$
\| u(t,\cdot +c_1 (t-m_1 (t)) )- p(\cdot) \|_{L^\infty
((-\infty,-C])} \leq \delta,
$$
$$
\| u(t,\cdot +c_N (t-m_N (t))  ) \|_{L^\infty ([C,
+\infty))} \leq \delta.$$
\end{enumerate}
\end{theorem}

The above theorem extends our previous convergence result in~\cite{DGM} to the case when the target terrace is not necessarily moving to the right. Some speeds~$c_k$ may be negative while others are positive (in which case the terrace splits into two opposite directions), or it may also happen that all speeds~$c_k$ are negative (in which case the terrace ``retreats" entirely). It is also possible that some speeds may be equal, in which case the limiting traveling waves still drift away from each other, as implied by the first part of statement $(ii)$. Note, however, that the theorem needs to avoid the possibility of zero speeds, the reason for which we will explain briefly below. The above theorem also relaxes the assumptions on the initial data compared with~\cite{DGM}, by not requiring $u_0$ to be of the Heaviside function type, but only assuming it to be steep enough.

This shows, as announced, the importance of the notion of terraces, in particular minimal ones, which succeed in capturing the dynamics of solutions of the Cauchy problem. It is not known, though, whether there also exist solutions converging in large time to a non minimal terrace. Although they do not directly answer this question, our next three theorems will yield some insight, by both providing new characterizations of minimal terraces, and showing that non minimal terraces may simply not exist.

Note that the theorem above, as well as those below, avoid the possibility of zero speeds occurring in terraces. Some study could still be conducted in such a situation, but reveals itself to be much more complicated and to lead to an even wider range of dynamics. The main reason is that stationary traveling waves do not provide a complete information, as they may not exist around any level set and any point in space, due to the space heterogeneity. Therefore, one must involve some other entire solutions, that cannot be parts of a propagating terrace as we defined it. We will discuss this in Section~\ref{sec:zero} and chose to focus, as far as our main results are concerned, on the non zero speeds case.\\

Let us conclude this section by turning to the question of uniqueness of minimal, semi-minimal or propagating terraces. This uniqueness will be meant here up to time shifts: that is, given two terraces $((p_k)_k,(U_k)_k)$ and $((p'_k)_k,(U'_k)_k)$, we say that they are equal up to time shifts if they share the same platforms and if, for any $k$, $U'_k$ is a time shift of $U_k$ which may depend on $k$. As stationary waves are obviously not unique up to time shifts, this clearly suggests that, as in our previous theorem, we will only consider terraces moving with non zero speeds, for instance those that we have already constructed in~\cite{DGM}.

Our three uniqueness theorems are the following, each one dealing with a different type of propagating terraces: 
\begin{theorem}[Uniqueness I: Minimal terraces]\label{th:unique_min}
If there exists a minimal propagating terrace $\mathcal{T}= \left( (p_k)_{0 \leq k \leq N},(U_k)_{1 \leq k \leq N} \right)$ with non zero speeds,
then it is the unique minimal propagating terrace up to time shifts.
\end{theorem}
\begin{proof} Theorem~\ref{th:unique_min} is actually almost trivial. All minimal terraces indeed share the same platforms, and one can easily check that, between each consecutive platforms, the steepest traveling wave is unique up to time shift, as long as its speed is not zero.
\end{proof}
\begin{theorem}[Uniqueness II: Semi-minimal terraces]\label{th:unique_semi}
If there exists a semi-minimal propagating terrace $\mathcal{T} = \left( (p_k)_{0 \leq k \leq N},(U_k)_{1 \leq k \leq N} \right)$ with non zero speeds, and if either of the following two statements is satisfied:
\begin{enumerate}[$(a)$]
\item the speeds of the $U_k$ are all different (i.e. $c_1 < c_2 < ... < c_N$); 
\item Assumption~\ref{assumption-speed1} is true;
\end{enumerate}
then it is the unique minimal terrace up to time shifts.

Under the Assumption~\ref{assumption-speed1}, it is also the unique semi-minimal terrace.
\end{theorem}
\begin{theorem}[Uniqueness III: Propagating terraces]\label{th:unique}
Under the Assumption~\ref{assumption-speed2}, if there exists a propagating terrace~$\mathcal{T}$ with non zero speeds, then it is the unique (hence minimal) propagating terrace up to time shifts.
\end{theorem}
%

Proving Theorems~\ref{th:unique_semi} and~\ref{th:unique} is slightly more intricate than for Theorem~\ref{th:unique_min}, and relies heavily on the link between steepness and moving speed. This link will be the subject of a result of independent interest, namely Theorem~\ref{th:lem_speeds}. This theorem will roughly state that, if $q$ is stable from below in a sense following from Assumptions~\ref{assumption-speed1} or~\ref{assumption-speed2}, then the steepest traveling wave connecting some other equilibrium to~$q$ is also the unique slowest. This is classical for non degenerate equilibria (by looking at the asymptotic behavior of waves as in~\cite{Hamel08}), but we prove it here with another method in order to deal with more general periodic nonlinearities.

\begin{remark}
In the homogeneous case, this relation between steepness and speed is always known to hold, so that Assumptions~\ref{assumption-speed1} and~\ref{assumption-speed2} are actually not needed~\cite{FifMcL77}. We believe that those hypotheses can also be relaxed in the periodic framework.
\end{remark}

On one hand, Theorem~\ref{th:unique_semi} deals with a somehow general case, where the restriction of the reaction nonlinearity between consecutive platforms may be of various types, including monostable ones. In such a situation, terraces cannot be expected to be unique (the speed is not even unique for standard monostable nonlinearities), but our result still provide a slightly easier to check characterization of the minimal terrace. On the other hand, under the assumptions of Theorem~\ref{th:unique}, all platforms of any terrace have to be strongly stable from both below and above, hence terraces ``skip" unstable platforms in some sense. One should look at it as a generalization of the well-known uniqueness of traveling waves up to shift in the standard bistable or ignition cases.

We remind that our underlying goal is to predict, in some sense, the shape of terraces, in particular minimal terraces as they seem to play the most important role. The following proposition shows how this may follow from our main results.
\begin{proposition}\label{proposition2} 
Let $p >q>r$ be $L$-periodic stationary solutions of \eqref{eqn1}, such that $q$ and $p$, $r$ and $q$ are connected in the sense of Definition~\ref{def:decomp}. Denote by $U$ a steepest traveling wave connecting $q$ to $p$, and $V$ a steepest traveling wave connecting $r$ to $q$, with respective speeds $c \neq 0 $ and $c' \neq 0$. 

Then:
\begin{itemize}
\item if $c < c'$, then $(U,V)$ is a minimal terrace,
\item if $c =c'$ and Assumption~\ref{assumption-speed1} holds, then $(U,V)$ is a minimal terrace.
\item if $c >c'$, then $p$ and $r$ are connected in the sense of Definition~\ref{def:decomp}.
\end{itemize}
\end{proposition}
Note that, from our Theorem~\ref{th:exists}, if two $L$-periodic stationary solutions $p>q$ are connected (in the sense of Definition~\ref{def:decomp} ), then there exists a steepest traveling wave among all traveling waves connecting $q$ to $p$, and it is unique up to time shift if its speed is not zero.

The first two statements of Proposition~\ref{proposition2} immediately follow from Theorem~\ref{th:unique_semi}. Then note that, by Corollary~\ref{th:exists_cor}, a minimal terrace should either have a single platform, or be made of two traveling waves connecting respectively $q$ and $p$, and $r$ and $q$. In the latter case, using again Theorem~\ref{th:unique_semi} (separately between $p$ and $q$, $q$ and $r$), these two waves should be $U$ and $V$, hence $c \leq c'$. The third statement of Proposition~\ref{proposition2} follows.

\paragraph{Plan of the paper} Our paper is organized as follows. We will begin in Section~\ref{sec:preliminaries} by recalling some properties on the steepness of solutions of \eqref{eqn1}, namely how it is preserved through time. Those properties rely heavily on the so-called intersection number (or zero number) argument. 

Then, we use in Section~\ref{sec:existence} those preliminaries to look at the large time behavior of solutions of \eqref{eqn1} with very steep initial data, such as of the Heaviside type. We will see that it converges to a minimal terrace, proving both our existence and convergence results. Section~\ref{sec:minspeed} will deal with an important result of independent interest, addressing the link between steepness and speed of pulsating traveling waves: namely, when the upper stationary state is stable, the steepest traveling wave is also the slowest. Our uniqueness results will quickly follow as corollaries, as we will show in Section~\ref{sec:uniqueness}.

After we have proved our main results, we will dedicate the two last sections on discussion in order to provide a better understanding of the notion of propagating terraces. First, in Section~\ref{sec:zero}, we will investigate the case of terraces moving with zero speed, that we mostly avoided in our theorems above. Dynamics will reveal themselves to be much more complex, involving new entire solutions that are not traveling waves in the standard sense, and preventing any uniqueness result to hold in general. Lastly, in Section~\ref{sec:exemples}, we will address a few examples ranging from standard cases to some multistable nonlinearities. Our goal will be to show how our results, in spite of their universal scope, can also easily be applied to particular cases.

\section{Steepness argument and fundamental lemma}\label{sec:preliminaries}

The proofs of our main results rely strongly on a zero-number
argument. The application of this argument --- or the
``Sturmian principle" --- to the convergence proof in semilinear
parabolic equations first appeared in \cite{Matano78}, and recently in~\cite{DGM} to prove not only the convergence but also
the existence of the target objects, namely the terrace and
pulsating traveling waves. This theory states that the number of sign changes of a solution of a second-order parabolic equation, or more generally, the number of intersections of two solutions, is nonincreasing in time. When looking at the steepness of solutions, we will only need to consider situations where only one intersection occurs. We refer the reader to \cite{An88,DuMatano10} for a better overview of the general argument and its applications.\\

We recall here two lemmas from the zero-number theory that we will need here. The first one is a corollary of the semi-continuity of the number of intersections with respect to pointwise topology.
\begin{lemma}\label{lem:steep_arg1}
Let $(w_n)_{n \in \mathbb{N}}: \R \to \R$ be a sequence of functions converging
to $w: \R \to \R$ pointwise on $\R$, and $v$ such that $w_n$ is steeper than $v$ for any $n \in \mathbb{N}$.

 Then $w$ is steeper than $v$.
\end{lemma}

The second lemma is a consequence of the fact that, as announced above, the number of intersections is nonincreasing in time, and also that the shape of an intersection remains the same as long as it still exists.
\begin{lemma}\label{lem:steep_arg2}
Let $u_1$ and $u_2$ be solutions of \eqref{eqn1} for nonnegative times, such that $u_1(0,x)$ is a piecewise continuous bounded function
on $\R$, steeper than $u_2(0,x)$ which is bounded and continuous on $\R$.

Then for any $t >0$, the function $x \mapsto u_1 (t,x)$ is steeper than $x \mapsto u_2 (t,x)$ and furthermore, unless $u_1 \equiv u_2$, any zero of $(u_1 - u_2) (t,\cdot)$ is non degenerate.
\end{lemma}

\begin{remark}
In both lemmas above, we compare steepness of functions which do not depend on time but on~$x \in \R$ only: even in Lemma~\ref{lem:steep_arg2}, where the functions $u_1$ and $u_2$ depend on time, their steepness is only compared at fixed times. This can easily be understood as a particular case of Definition~\ref{def:steep}.
\end{remark}

The non degeneracy of any zero of $u_1 - u_2$ follows from the fact that, whenever a degenerate zero of $u_1 - u_2$ appears for some $t$, the number of sign changes is decreasing in time around $t$. This implies, if $t>0$, that $u_1 -u_2$ changes sign at least twice for anterior times, which contradicts the property that $u_1$ is steeper than $u_2$ for any given time. We refer the reader to \cite{DGM} for a more detailed proof of Lemma~\ref{lem:steep_arg2} (note that the same argument applies to the letter here, even though the function $a$ may not be constant).\\

Before we go on, let us recall a definition of the $\omega$-limit set of a solution~$u$. This is slightly different from the standard one, as we add
arbitrary spatial translations while taking the long time limit.
The reason for adopting this definition is that, since each step
of a terrace (that we want to capture by looking at the asymptotic profile of $u$) moves at a different speed, we will need multi-speed
observations to be able to proceed.

\begin{definition}\label{def:omegalimit}
Let $u(t,x)$ be a bounded solution of Cauchy problem
\eqref{eqn1}. We call $v(t,x)$ an
\textbf{$\boldsymbol{\omega}$-limit orbit} of $u$ if there exist
two sequences $t_j \rightarrow +\infty$ and $k_j \in \mathbb{Z}$
such that
\[
u(t+t_j,x+k_j L) \rightarrow v(t,x)\ \mbox{ as } j \to +\infty
\; \mbox{ locally uniformly on } \R^2.
\]
\end{definition}

\begin{remark}
By parabolic estimates, the above convergence takes place in
$C^2$ in $x$ and $C^1$ in $t$.  Hence
one can easily check that any $\omega$-limit orbit of $u$ is
an entire solution of \eqref{eqn1}. Moreover, if $v(t,x)$ is an
$\omega$-limit orbit of $u$, then so is $v(t+\tau,x+kL)$ for any
$\tau \in \R$ and $k\in \mathbb{Z}$, as well as any limit as $\tau \rightarrow \pm \infty$ or $k \rightarrow \pm \infty$, when it exists.
\end{remark}

Let us now state the fundamental lemma that will be used repeatedly
throughout our paper, and whose proof will easily follow from the properties we recalled above:

\begin{lemma}\label{lem:omega-steep}

Let $u(t,x)$ be any piecewise continuous and bounded solution of Cauchy problem
\eqref{eqn1} with an initial datum $0 \leq u_0 \leq p$. We assume that $u_0$ is steeper than any entire solution of \eqref{eqn1} lying between 0 and $p$.

Then any $\omega$-limit orbit $v$ of $u$ is steeper than any other entire
solution of~\eqref{eqn1} lying between $0$ and~$p$.
\end{lemma}

\begin{remark}
Note that it is not clear a priori which initial data are steeper than any other entire solution lying between 0 and $p$. However, one can immediately see that such a property is always satisfied with an Heaviside type initial datum, that is a function identically equal to $p$ on the interval $(-\infty,X]$, and identically equal to 0 on $(X,+\infty)$, for some $X \in \R$.
\end{remark}

\begin{proof}
Let the sequences $t_j \rightarrow + \infty$
and $k_j \in \mathbb{Z}$ be such that $u(t+t_j,x+k_jL)\to v (t,x)$
locally uniformly as $j \to + \infty$. 

Let $w$ be any entire solution lying between 0 and $p$, and any $t,t' \in \R$. Consider $j$ large enough so that $t_j > -t'$. By our assumption, we have that $x \mapsto u_0 (x)$ is steeper than $x \mapsto w (-t_j -t' + t, x)$, hence $x \mapsto u(t_j +t',x+k_j L)$ is steeper than $x\mapsto w (t,x)$ using Lemma~\ref{lem:steep_arg2}. Using Lemma~\ref{lem:steep_arg1}, one can then pass to the limit as $j \rightarrow +\infty$ to get that $x\mapsto v(t',x)$ is steeper than $x \mapsto w (t,x)$, where $t$ and $t'$ are arbitrary. In other words, $(t,x) \mapsto v (t,x)$ is steeper than $(t,x) \mapsto w (t,x)$ in the sense of Definition~\ref{def:steep}.
\end{proof}

\section{Existence and convergence to a minimal terrace}\label{sec:existence}

In this section, we will prove our existence Theorem~\ref{th:exists}. As mentioned before, we only need to prove that if there exists a decomposition connecting 0 and~$p$, then there exists a minimal terrace connecting 0 to $p$, with the additional powerful property that each of its parts is steeper than any other entire solution of~\eqref{eqn1}.

The sketch of the proof is to exhibit a minimal terrace as a limiting profile of a solution associated to an Heaviside initial datum, which is the steepest of functions between~0 and~$p$. As shown in the previous section, such extreme steepness should be preserved through time and even at the limit as $t \rightarrow +\infty$. More precisely, by Lemma~\ref{lem:omega-steep}, we know that any $\omega$-limit orbit is steeper than any other entire solution. The fact that there exists a decomposition between 0 and~$p$ will then insure that those limits must be either front-like or spatially periodic stationary solutions.

Once we know that such a minimal terrace exists, one will easily be able to check that the proof can be extended to any steep enough initial datum. Provided that the minimal terrace is unique, we then get a full convergence result toward it, namely Theorem~\ref{th:CV}.

\subsection{Existence of a minimal terrace}

Denote by $(q_k)_{0 \leq k \leq N}$ the decomposition between 0 and $p$, and for each $k$, by~$V_k$ a pulsating traveling wave connecting $q_k$ to $q_{k-1}$.

Let $u$ be the solution of \eqref{eqn1} with initial datum
$$u_0 (x)= H(-x) p(x),$$
where $H$ is the Heaviside function. We already know from Lemma~\ref{lem:omega-steep} that any $\omega$-limit $v$ of $u$ is steeper than any other entire solution lying between $0$ and $p$. In particular, $v$ is steeper than any of its time shift, which easily implies that either $\partial_t v \equiv 0$, $\partial_t v >0$ or $\partial_t v <0$.\\

First, thanks to standard parabolic estimates, take a sequence $t_k \rightarrow +\infty$ such that $u(t+t_k,x)$ converges locally uniformly to some $\omega$-limit $w (t,x)$. One can assume without loss of generality that it is stationary, that is $w(t,x) \equiv w(x)$ does not depend on time. Indeed, if not, one can use the fact that $w$ is monotonically increasing or decreasing with respect to time, so that $w(+\infty,x)$ exists and is also an $\omega$-limit of $u$ in the non moving frame. Because $w$ is steeper than the $q_k$ and $V_k$, it follows that~$w$ is a stationary wave connecting some $q_{i}$ to $q_{j}$ with $i \geq j$ (if $i =j$, then it is trivial and $w \equiv q_{i}$).

Our goal is to construct two terraces, connecting respectively $0$ to $q_{i}$ with nonnegative speeds, and $q_{j}$ to $p$ with nonpositive speeds.\\

Assume that $q_{i} >0$. Let $\alpha_1$ such that $q_{i +1} (0) < \alpha_1 < q_{i} (0)$. One can then define, for any $n \in \mathbb{N}$:
$$\tau_n (\alpha_1) = \min \{ t \geq 0 \; | \ u (t,nL) =\alpha_1 \}.$$
Let now the $\omega$-limit
$$w_\infty (t,x;\alpha_1) = \lim_{n\rightarrow +\infty} u(t+\tau_n (\alpha_1),x+nL).$$
Let us check that this limit is well-defined. Note that the family of functions $\{ u(t+\tau_n (\alpha_1),x+nL) \}_{n \in \mathbb{N}}$ is relatively compact, so that we can assume that it converges, up to extraction of some subsequence, to some $w_\infty (t,x;\alpha_1)$, which is the unique steepest entire solution of \eqref{eqn1} satisfying $w_\infty (0,0;\alpha_1) =\alpha_1$. Hence, $w_\infty$ does not depend on the choice of the subsequence. It follows that the whole sequence indeed converges to $w_\infty$ as $n \rightarrow +\infty$.

Moreover, from the definition of $\tau_n (\alpha_1)$ and the monotonicity of any $\omega$-limit, we have that $\partial_t w_\infty \geq 0$. One can also note that 
$$u(t,x+L) < u(t,x),$$
hence
$$w_\infty (t,x+L) \leq w_\infty (t,x),$$
for any $(t,x) \in \R^2$. This follows from the comparison principle and the fact that $u(t,\cdot+L)$ is the solution of \eqref{eqn1} with initial datum $H(-L-x) p(x) \leq H(-x) p(x)$. It immediately implies that the sequence $\tau_n (\alpha_1)$ is increasing. Let us now prove the following claim:
\begin{claim}
The sequence $\tau_{n+1} (\alpha_1) - \tau_n (\alpha_1)$ converges to either a positive constant $T>0$, or to $+\infty$. 
\end{claim}
\begin{proof} Assume that $\tau_{n+1} (\alpha_1) - \tau_n (\alpha_1)$ does not converge to $+\infty$. Then, up to extraction of some subsequence, $\tau_{n+1} (\alpha_1) - \tau_n (\alpha_1) \rightarrow T \geq 0$ as $n \rightarrow +\infty$. Therefore, for any $(t,x) \in \R^2$:
\begin{eqnarray*}
w_\infty (t+T,x+L;\alpha_1)&=&\lim_{n\rightarrow +\infty} u(t+\tau_n (\alpha_1) + \tau_{n+1} (\alpha_1) - \tau_n(\alpha_1),x+(n+1)L)\\
&=&\lim_{n\rightarrow +\infty} u(t+ \tau_{n+1} (\alpha_1),x+(n+1)L)\\
&=&w_\infty (t,x;\alpha_1).
\end{eqnarray*}
If $T=0$, this means that $w_\infty$ is periodic in space. This is not possible. Indeed, one would then have that $w_\infty (0,nL;\alpha_1) =\alpha_1$ for any $n \in \mathbb{Z}$, which contradicts the fact that it is steeper than $V_{i +1}$ connecting $q_{i +1}$ to $q_{i}$. On the other hand, if $T >0$, then it is a monotonically increasing pulsating traveling wave, with speed $c = \frac{L}{T}$. By uniqueness of the limit and of the speed of a traveling wave, it follows that the whole sequence $\tau_{n+1} (\alpha_1) -\tau_n (\alpha_1)$ converges to $T$, which concludes the proof of this claim.\end{proof}\\

We have shown, along with the above claim, that when $\tau_{n+1} (\alpha_1) - \tau_n (\alpha_1) \rightarrow T>0$, then $w_\infty$ is a pulsating wave. As it is steeper than any other entire solution of \eqref{eqn1}, and from the choice of $\alpha_1$, one can easily check that it connects some $q_{i_1}$ to $q_{i_0}$ with $i_1 > i$ and $i_0 \leq i$. As $w_\infty$ and $w$ (obtained above as an $\omega$-limit in the non moving frame) are both steeper than each other, they cannot intersect, nor can any of their translates by some space shift $kL$ with $k\in \mathbb{Z}$. Thus, one in fact has that $i_0 = i$. 

Now, consider the case when $\tau_{n+1} (\alpha_1) - \tau_n (\alpha_1)$ tends to $+\infty$. This means that 
$$w_\infty (t,L;\alpha_1) \leq \alpha_1, \ \forall t \geq 0.$$
By its monotonicity and boundedness, we know that $w_\infty$ converges to a stationary solution as $t \rightarrow +\infty$, such that $w_\infty (+\infty,0;\alpha_1) \geq \alpha_1$ and $w_\infty (+\infty,L;\alpha_1) \leq \alpha_1$, and which is still steeper than any other entire solution of \eqref{eqn1}. From this and the fact that $w_\infty (+\infty,x+L;\alpha_1) \leq w_\infty (+\infty,x;\alpha_1)$ for any $x\in \mathbb{R}$, one can check as before that $w_\infty (+\infty,x;\alpha_1)$ converges to some $q_{i_1}$ with $i_1>i$ as $x\rightarrow +\infty$, and to $q_{i}$ as $x \rightarrow -\infty$. Hence, it is a pulsating traveling wave with speed 0 connecting~$q_{i_1}$ to~$q_i$.

In both cases, we have obtained a pulsating traveling wave with nonnegative speed connecting some $q_{i_1}$ with $i_1 >i$ to $q_{i}$. If $q_{i_1}$ is positive, we can reiterate by choosing $q_{i_1 +1} < \alpha_2 < q_{i_1} (0)$, and find a pulsating traveling wave connecting some $q_{i_2}$ with $i_2 > i_1$ to $q_{i_1}$. The iteration stops when $q_{i_n} \equiv 0$ for some $n \in \mathbb{N}$. This clearly happens in a finite number of steps, from the fact that there is a finite number of $q_k$. Therefore, we obtain a decreasing sequence of periodic stationary solutions $(p_k)_{0 \leq k \leq N_1}$, and for each $k$ a pulsating traveling wave $U_k$ connecting $p_{k+1}$ to $p_k$. Furthermore, from the construction of the $U_k$ above, we also have that
\begin{equation}\label{eqn:speed_ck}
c_k = \lim_{n \to +\infty} \frac{L}{\tau_{n+1} (\alpha_k) - \tau_n (\alpha_k)} = \lim_{n \rightarrow +\infty} \frac{nL}{\tau_n (\alpha_k)}
\end{equation}
for some well-chosen decreasing and finite sequence $\alpha_k$.

Since $\tau_n (\alpha)$ is increasing with respect to $\alpha$, the sequence $c_k$ is nondecreasing. Finally, we have constructed a propagating terrace connecting 0 to $q_{i}$ with nonnegative speeds.\\

One can proceed similarly to get a propagating terrace connecting $q_{j}$ to $p$ with nonpositive speeds. We immediately get a propagating terrace connecting~0 to~$p$ by combining both terraces constructed above, as well as the stationary wave $w$ when it is not trivial. We end the proof of Theorem~\ref{th:exists} by noting that each part of this terrace is steeper than any other entire solution, as follows from the fact that each of those parts is also an $\omega$-limit orbit of the solution of \eqref{eqn1} with Heaviside type initial datum. One can easily check that this also immediately implies minimality of the terrace.

\subsection{Convergence to the unique minimal terrace}

Let us now proceed to the proof of Theorem~\ref{th:CV}, and assume that there exists a minimal terrace $\mathcal{T} =((p_k)_k,(U_k)_k)$ such that for any $k$, the speed $c_k$ of $U_k$ is not equal to 0. According to Theorem~\ref{th:unique_min} (that we already proved in Section~\ref{sec:intro}), the minimal terrace is unique up to time shifts. Moreoever, thanks to Theorem~\ref{th:exists}, we also know that each of the $p_k$ and $U_k$ is steeper than any other entire solution lying between~0 and~$p$.

Let $u$ be the solution of \eqref{eqn1} with initial datum $u_0$. We recall that we assume that~$u_0$ satisfies the limiting conditions $u(x)-p(x) \rightarrow 0 $ as $x\rightarrow -\infty$, and $u(x) \rightarrow 0$ as $x \rightarrow +\infty$, and also that $u_0$ is steeper than any $p_k$ and $U_k$. First note that the latter assumption implies that $u_0$ is steeper than any entire solution lying between 0 and~$p$. Indeed, let some entire solution $v(t,x)$ of \eqref{eqn1}, lying between 0 and~$p$, such that for some $t$ and $x_1$, $u_0 (x_1) = v(t,x_1)$. Then, there exists some $k$ such that either $p_k (x_1)=u_0 (x_1)$, or $p_{k+1} (x_1) < u_0 (x_1) < p_k (x_1)$. In the former case, as $u_0$ is steeper than $p_k$ and $p_k$ is steeper than $v$, one gets that $u_0 (x) \geq v (t,x)$ for any $x \leq x_1$, and $u_0 (x) \leq v(t,x)$ for any $x \geq x_1$. In the latter case, there exists some $t'$ such that $U_k (t',x_1)=u_0 (x_1)=v(t,x_1)$, and one can again conclude by the same argument that $u_0$ is steeper than $v$.\\

We now turn back to the proof of Theorem~\ref{th:CV}, which will share many similarities with the previous subsection. The difficulty here is to make sure that the $\omega$-limit of $u$ around a given level set is unique, which is necessary to get the wanted convergence. As before, we know by Lemma~\ref{lem:omega-steep} that any $\omega$-limit $v$ of~$u$ is steeper than any other entire solution lying between $0$ and $p$, and that it satisfies either $\partial_t v \equiv 0$, $\partial_t v >0$ or $\partial_t v <0$.\\

We first prove the following claim:
\begin{claim}\label{claim_21} The function $u(t,x)$ converges locally uniformly to the unique $p_i$ such that $c_k <0$ for all $k \leq i$ and $c_k >0$ for all $k>i$.
\end{claim}

\begin{proof} 
First, take some sequence $t_j \rightarrow +\infty$ such that $u(t+t_j,x)$ converges locally uniformly to some $\omega$-limit $w (t,x)$. Recall that $w$ is either stationary, monotonically increasing or decreasing in time. 

If it is monotonically increasing, we can define $w_1(x):= w(-\infty,x) < w_2 (x) := w(+\infty,x)$, both being also $\omega$-limits of $u$ in the non moving frame. Let now $w_1 (0) < \alpha := w(0,0) < w_2 (0)$. Thus, there exists a sequence $t'_j \rightarrow +\infty$ such that for any $j \in \mathbb{N}$, $u(t'_j,0)=\alpha$ and $\partial_t u(t'_j,0) \leq 0$. Up to extraction of some subsequence, the sequence $u(t+t'_j,x)$ converges to an $\omega$-limit $v$ such that $v(0,0)=\alpha=w(0,0)$ and $\partial_t v(0,0) \leq 0 < \partial_t w (0,0)$. This is a contradiction as $v$ and $w$ are both steeper than each other. Similarly, one can prove that $w$ is not monotonically decreasing in time either, hence it is stationary.

Moreover, there exists some $i$ such that either $w(0) =p_i (0)$, or $w(0)=U_i (t,0)$ for some $t \in \mathbb{R}$. Since $w$, $p_i$ and $U_i$ are all steeper than each other, it follows that either $w \equiv p_i$ or $w \equiv U_i$. The latter is not possible because $U_i$ is not stationary by assumption. Hence, $w \equiv p_i$. Using the same method as in the previous section, one can construct a minimal terrace $\mathcal{T}'$ including the platform~$p_i$, so that any traveling wave below $p_i$ moves with nonnegative speed, and any wave above $p_i$ moves with nonpositive speed. By uniqueness, one immediately gets that $\mathcal{T}$ and $\mathcal{T}'$ are the same terrace up to time shifts, thus $i$ is the unique index such that $c_k < 0$ for all $k \leq i$ and $c_k >0$ for all $k>i$.
In particular, the limit $w$ did not depend on the choice of the sequence $t_j \rightarrow +\infty$, so that $u(t,x)$ converges locally uniformly to $p_i$ as $t \rightarrow +\infty$.
\end{proof}\\

Let us now go back to the proof of Theorem~\ref{th:CV}. Take, for any $k$, $\alpha_k \in (p_k (0),p_{k-1} (0))$. Then define, for any $j \in \mathbb{N}$:
$$\tau_j (\alpha_k) = \left\{
\begin{array}{l}
\min \{ t \geq 0 \; | \ u (t,jL) =\alpha_k \} \mbox{ if } k > i\vspace{3pt},\\
\min \{ t \geq 0 \; | \ u (t,-jL) =\alpha_k \} \mbox{ if } k\leq i.
\end{array}
\right.
$$
Note that those are well-defined, at least for $j$ large enough, from the assumption that $u_0 - p \rightarrow 0$ as $x \rightarrow -\infty$ and $u_0 \rightarrow 0$ as $x \rightarrow +\infty$, along with the locally uniform convergence to $p_i$ proved above.

As in the previous section, we know that the following $\omega$-limit
$$w_\infty (t,x;\alpha_k) = \lim_{j\rightarrow +\infty} u(t+\tau_j (\alpha_k),x+jL)$$
is well defined. Moreover, it is steeper than $U_k$ and conversely (recall that $U_k$ is steeper than other entire solutions). As $U_k$ moves with non zero speed, there exists some time $t$ such that $U_k (t,0)=\alpha_k = w_\infty (0,0;\alpha_k)$. We conclude that $U_k \equiv w_\infty (\cdot,\cdot;\alpha_k)$ up to some time shift, that is, without loss of generality:
$$U_k (t,x) = \lim_{j\rightarrow +\infty} u(t+\tau_j (\alpha_k),x+jL),$$
where the convergence holds locally uniformly in time and space.\\

We can now proceed to the proof of the wanted convergence. The proof is the same as in \cite{DGM}, but we include it for the sake of completeness. Let us first deal with the locally uniform convergence to the pulsating traveling waves $U_k$ with $c_k \neq 0$ along the moving frames with speed $c_k$ and some sublinear drifts. Fix some $k$ such that $c_k >0$ and, for any large~$t$, define $j(t) \in \mathbb{N}$ such that
$$j(t) \frac{L}{c_k} \leq t < \left(j(t) +1\right) \frac{L}{c_k},$$
and $m_k (t)$ the piecewise affine function defined by
$$m_k (t)= \tau_{j(t)} (\alpha_k) -t \ \ \text{   if   } \ t =j(t)\frac{L}{c_k}.$$
Recalling~\eqref{eqn:speed_ck}, the sequence $\left\{\frac{\tau_{j} (\alpha_k)}{j}\right\}_j$ converges to $\frac{L}{c_k}$, so that $m_k(t) = o(j(t))=o(t)$ as $t \to +\infty$.

Furthermore, since
$$U_k (t,x)= w_\infty (t,x;\alpha_k) = \lim_{j \rightarrow +\infty} u (t+\tau_{j} (\alpha_k), x +jL)$$
where the convergence is understood to hold locally uniformly with respect to $(t,x)\in\mathbb R^2$, and since $t+m_k (t)-\tau_{j(t)} (\alpha_k) \sim (t-j(t) \frac{L}{c_k}) = O_{t \rightarrow +\infty} (1)$ and $x+c_k t - j(t) L$ stay bounded, one can check that
$$u (t+m_k (t),x+c_k t) - U_k \left(t-j(t) \frac{L}{c_k},x - j(t)L +c_k t  \right) \to 0 \ \ \text{ as } \ t \to +\infty.$$
Thus,
$$u (t,x+c_k (t-m_k (t))) - U_k \left( t- m_k (t),x+c_k (t -m_k (t) )\right) \to 0 \ \ \text{ as } \ t \to +\infty, $$
wherein both of the above convergences hold locally uniformly with respect to~$x \in \R$. The case $c_k <0$ can easily be dealt with the same way.\\

It now remains to consider what happens ``outside" of the moving frames with speed $(c_k)_{1 \leq k \leq N}$. We first claim the following monotonicity property:
\begin{equation}\label{eqn:mon}
u (t,x) \geq u (t,x+L), \mbox{ for any } x \in \R \mbox{ and } t \geq 0.
\end{equation}
From the periodicity of the equation and the parabolic comparison principle, one only has to check the above inequality for $t =0$. Let any $x \in \R$, and let $k$ such that either $u_0 (x)=p_k (x)$ or $u_0 (x) = U_k (t,x)$ for some $t \in \R$. Since $u_0$ is steeper than $p_k$ and $U_k$, one has that either 
$$u_0 (x+L) \leq p_k (x+L)= p_k (x) = u_0 (x),$$
or
$$u_0 (x+L) \leq U_k (t,x+L) \leq U_k (t,x)=u_0 (x).$$
We conclude, as announced, that the inequality \eqref{eqn:mon} is satisfied.

Let us go back to the proof of convergence ``outside" the moving frames. We first look near $-\infty$, that is when $x+ c_1 (t-m_1 (t)) \to -\infty$. In that case, we use the fact that
\begin{equation}\label{eqn:asymp0}
\lim_{x \rightarrow -\infty} U_1 (t,x+c_1t) - p(x) = 0,
\end{equation}
uniformly with respect to time. Let any small $\delta >0$, and $x_\delta$ such that for all~$t$:
$$ p(x)-\frac{\delta}{2} \leq U_1(t-m_1 (t),x+c_1 (t-m_1 (t))) \leq p(x) \text{  for all } x\leq -x_\delta +L.$$
Let now $t$ large enough so that for $x\in [-x_\delta , x_\delta]$, we have that $$| u (t,x+c_1 (t-m_1 (t))) - U_1 (t-m_1 (t),x+c_1 (t-m_1 (t)))| \leq \frac{\delta}{2}.$$
Then, using \eqref{eqn:mon}, one gets for all large $t$:
$$ p(x)-\delta \leq u (t,x+c_1 (t-m_1 (t))) \leq p(x) \text{  for all } x\leq -x_\delta +L.$$
One can proceed similarly to get that for any $\delta >0$, there exists $C$ such that for any $x \geq C$,
$$ | u (t,x+c_N (t-m_N (t)))  | \leq \delta.$$
Lastly, let any integer $1 \leq k \leq N$. Then
$$
\lim_{x \rightarrow +\infty} U_k (t,x+c_k t) - p_k (x) =0 \mbox{  and  }
\lim_{x \rightarrow -\infty} U_{k+1} (t,x+c_{k+1}t ) - p_k (x) =0,
$$
where the convergences are uniform in time.
As above, one can use \eqref{eqn:mon} to show that there exists $C$ such that for large time $t$:
$$ p_k(x)+\delta \geq u (t,x+c_k (t-m_k (t)))\text{  for all } x\geq C.\vspace{3pt}$$
$$ p_k (x)-\delta \leq u (t,x+c_{k+1} (t-m_{k+1} (t))) \text{  for all } x\leq -C.$$
This ends the proof of Theorem~\ref{th:CV}.

\begin{remark}
In general, even when zero speeds may occur, there still exists some minimal and steepest terrace thanks to our existence Theorem~\ref{th:exists}. However, it is no longer unique. Although one could proceed as above to get a similar convergence result on the level sets with positive or negative speeds, the zero speed case is much more difficult to describe and would involve a wider range of entire solutions which may not be stationary waves.
\end{remark}

\section{Minimality of the speeds of the minimal terrace}\label{sec:minspeed}

In this section, we first show a theorem of independent interest, namely Theorem~\ref{th:lem_speeds} below, stating that the steepest traveling wave is also the slowest under some strong stability assumption on the upper state. We will then give as a corollary the minimality of the speeds of the minimal terrace. Furthermore, as we will see in the next section, most of our uniqueness properties on terraces will also follow from Theorem~\ref{th:lem_speeds}.

\subsection{Steepness and minimality of the speeds}

Let us state the fundamental theorem of this section:
\begin{theorem}\label{th:lem_speeds}
Let two traveling waves $U_1$ and $U_2$ connecting respectively $q_1$ and $q_2$ to $q$. Assume that there exist $\delta >0$ and $g$ such that $\mu_g \leq 0$ and $\partial_u f(x,u) \leq g$ for all $x \in \R$ and $u \in [q (x) - \delta, q(x)]$, and that $U_1$ is steeper than $U_2$.

Then $c_1 \leq c_2$, where $c_1$ and $c_2$ are the speeds of respectively $U_1$ and $U_2$.

Moreover, if $c_1 = c_2 \neq 0$, then $U_1 \equiv U_2$ up to some time shift.
\end{theorem}

This means that, whenever the upper state is strongly stable in some sense (including but not limited to linear stability), the steepest traveling wave is the slowest and, if its speed is not zero, it is even the unique slowest traveling wave. 

Moreover, one can also easily check by turning the problem upside down (more precisely, by looking at the traveling waves $-U_1 (t,-x)$ and $-U_2 (t,-x)$), that if it is the lower state which is strongly stable, then the steepest traveling wave is the fastest. Therefore, in the bistable case, this already generalizes to the periodic framework the uniqueness of the traveling wave, which was well known for the homogeneous problem~\cite{AW78,FifMcL77}.\\

Before we prove Theorem~\ref{th:lem_speeds}, we immediately apply it to minimal terraces to get that, since they are the steepest, they are also the slowest in some sense:
\begin{theorem}\label{th:min_speeds}
Let Assumption~\ref{assumption-speed1} be satisfied. Assume that there exists a minimal terrace $\mathcal{T}=((p_k)_k,(U_k)_k)$ and for each $k$, let $c_k$ the speed of $U_k$. Then:
\begin{enumerate}[$(i)$]
\item If $c_k >0$, then $U_k$ is the unique traveling wave with minimal speed among all traveling waves connecting some $q < p_{k-1}$ to $p_{k-1}$.
\item If $c_k <0$, then $U_k$ is the unique traveling wave with maximal speed among all traveling waves connecting $p_{k}$ to some $q> p_k$.
\item If $c_k =0$, then $U_k$ is a traveling wave with either minimal or maximal speed among all traveling waves connecting $p_{k}$ to $p_{k-1}$.
\end{enumerate}
\end{theorem}

\begin{proof} 
Let us first note that, when $c_k \neq 0$, the traveling wave $U_k$ is not only steeper than any other traveling wave connecting $p_{k}$ to $p_{k-1}$, but also than any other entire solution of \eqref{eqn1}, and in particular than any traveling wave connecting some $q < p_{k-1}$ to $p_{k-1}$, or $p_k$ to some $q > p_k$. Indeed, we know from Theorem~\ref{th:exists} that there exists a minimal terrace $\mathcal{T}' = ((p_k)_k,(U'_k)_k)$ such that each $U'_k$ is steeper than any other entire solution. Up to some time shift and as $c_k \neq 0$, we can assume without loss of generality that $U_k$ and $U'_k$ intersect and, as they are steeper than each other, we get that $U'_k \equiv U_k$. As announced, this implies that~$U_k$ is steeper than any other entire solution.

Before going back to the proof of Theorem~\ref{th:min_speeds}, let us note that the role of Assumption~\ref{assumption-speed1} is to insure that the hypotheses of Theorem~\ref{th:lem_speeds} are satisfied. This is the subject of the following claim:
\begin{claim}\label{claim:42}
If there exists a pulsating traveling wave connecting some $q_2$ to $q_1$ with positive speed, then $q_1$ is isolated and stable from below.
\end{claim}
\begin{proof}[Proof of Claim~\ref{claim:42}]
We briefly sketch the proof, which is similar to that of Lemma~4.3 in~\cite{DGM}. We first define, for any $\lambda >0$, $\mu (\lambda)$ the principal eigenvalue of 
$$\left\{
\begin{array}{rcl}
-\partial_x (a \partial_x \phi_\lambda ) + 2 \lambda \partial_x \phi_\lambda - \frac{\partial f}{\partial u} (x, q_1 (x)) \phi_\lambda & = & \mu (\lambda) \phi_\lambda \ \mbox{ in } \mathbb{R},\vspace{3pt}\\
\phi_\lambda >0 \ \mbox{ and periodic.}
\end{array}
\right.$$
Adapting a formula from Nadin~\cite{Nadin10}, it is known that
$$
\begin{array}{rcl} \mu (\lambda ) &=&  \displaystyle \min_{\eta \in C^1_{per} , \eta >0} \frac{1}{\int_0^L \eta^2 } \left(  \int_0^L a \eta'^2 - \int_0^L \frac{\partial f}{\partial u} (x,q_1) \eta^2 \right. \vspace{3pt} \\
& & \displaystyle \hspace{3cm} \left.  + \lambda^2 \bigg( \int_0^L a \eta^2 - \frac{L^2}{\int_0^L a^{-1} \eta^{-2}} \bigg)\right).
\end{array}
$$
It in particular follows that $\mu (\lambda) - \mu (0) = O (\lambda^2)$ for small enough $\lambda$. 

Now proceed by contradiction and assume that $q_1$ is not isolated from below. Then there exists a sequence of stationary solutions $r_j \leq q_1$, which tends to $q_1$ as $j \to +\infty$. This implies that $\mu (0)=0$ ($q_1$ is a degenerate equilibrium), thus $\mu (\lambda)= O (\lambda^2)$ for small $\lambda$. It is then straightforward to check that, for $j$ large enough and $c$ arbitrary small, one can construct a supersolution of \eqref{eqn1} of the type
$$\min \{ q_1 (x) , e^{-\lambda (x-ct)} \phi_\lambda (x) + r_j (x) \},$$ 
where $\lambda >0$. This easily contradicts the fact that there exists a pulsating traveling wave connecting some $q_2$ to $q_1$ with positive speed.

A similar argument leads to the same contradiction if $q_1$ is unstable from below. The claim is proved.\end{proof}\\

We now return to the proof of Theorem~\ref{th:min_speeds}. Thanks to Claim~\ref{claim:42} and Assumption~\ref{assumption-speed1}, we can apply Theorem~\ref{th:lem_speeds} to immediately conclude that statement $(i)$ of Theorem~\ref{th:min_speeds} holds. Statement $(ii)$ is very similar, by applying the same argument to $-U_k (t,-x)$. 

Statement $(iii)$ follows from the fact that if there exist two traveling waves $V_-$ and $V_+$ connecting $p_k$ to $p_{k-1}$ with speeds $c_- <0$ and $c_+>0$, then $p_k$ is isolated and stable from above, and $p_{k-1}$ isolated and stable from below. Using Assumption~\ref{assumption-speed1} and Theorem~\ref{th:lem_speeds}, we get that $c_k$ is both the minimal and the maximal speed among traveling waves connecting $p_k$ to $p_ {k-1}$, which is a clear contradiction. This ends the proof of Theorem~\ref{th:min_speeds}.\end{proof}

\subsection{Proof  of Theorem~\ref{th:lem_speeds}}

We begin by assuming that either $c_2 < c_1$, or $c_2 = c_1 \neq 0$. Up to some time shift, thanks to the asymptotics of $U_1$ and the characterization of pulsating traveling waves, one can find $0 <\delta' < \delta$ and $x_0$ such that
$$U_1 (t,x) \in \left[q (x)-\delta,q (x) \right] \; \mbox{ for all } x \leq x_0 + c_1 t,$$
$$U_1 (t,x) \in \left[q (x)-\delta,q (x) - \delta' \right] \; \mbox{ for all } x = x_0 + c_1 t.$$
Let us also assume that
$$U_2 (t,x) \in \left[q (x)-\frac{\delta'}{2},q (x) \right] \; \mbox{ for all } t \leq 0 \; \mbox{ and } x \leq x_0 + c_1 t .$$
This clearly holds whenever $c_1 \geq c_2 \neq 0$, up to some time shift of $U_2$. When $c_2 = 0 <c_1$ one can achieve the same inequality without loss of generality, by shifting $U_1$ in time so that $x_0$ is a large enough negative number.

We have in particular that, for any $t \leq 0$, \begin{equation}\label{compx0}
U_1 (t,x_0+c_1 t) \leq U_2 (t,x_0+c_1 t) - \frac{\delta'}{2}. \end{equation}
We now place ourselves in the moving frame with speed $c_1$, and define
$$\eta (t,x) := U_2 (t, x +c_1 t) - U_1 (t,x+c_1 t).$$
One can check from the above that the function $\eta$ satisfies:
\begin{equation}
\left\{
\begin{array}{rl}
\displaystyle \partial_t \eta = \partial_{x} (a (x) \partial_x \eta ) + c_1 \partial_x \eta + h(t,x) \eta, & \forall (t,x) \in (-\infty,0) \times (-\infty,x_0), \vspace{5pt}\\
\displaystyle \eta (t,x_0) \geq \frac{\delta'}{2} >0, & \forall t \in (-\infty,0),\vspace{5pt}\\
\displaystyle \lim_{x \rightarrow -\infty} \eta (t,x) = 0, & \forall t \in (-\infty,0),
\end{array}
\right.
\end{equation}
where
\begin{eqnarray*}
h(t,x)& := & \left\{
\begin{array}{lc}
\displaystyle \frac{ f(x+c_1 t,U_2 (t, x +c_1 t)) -f(x+c_1 t,U_1 (t, x +c_1 t))}{\eta (t,x)} & \mbox{ if } \eta (t,x) \neq 0 , \vspace{5pt}\\
\displaystyle  \partial_u f(x+c_1 t,U_1 (t,x+c_1 t))& \mbox{ if } \eta (t,x) =0 .
\end{array}\right.
\end{eqnarray*}

We now first proceed by contradiction and assume that $c_2 < c_1$. Therefore, we have that
$$ \liminf_{t \rightarrow -\infty} \inf_{x \leq x_0} \eta (t,x)  \geq 0.$$
Recall that $\mu_g \leq 0$ and $\phi$ are respectively the principal eigenvalue and principal eigenfunction of the periodic problem:
\begin{equation*}
\left\{
\begin{array}{l}
\displaystyle \partial_{x} (a \partial_x \phi ) +g \phi  =  \mu_g \phi \ \mbox{ in } \R, \vspace{5pt}\\
\phi >0 \mbox{ and  $L$-periodic}.
\end{array}
\right.
\end{equation*}
It follows from our assumptions and the choice of $x_0$ that  $h(t,x) \leq g (x+c_1 t)$ for any $t \leq 0$ and $x \leq x_0$.
Then for any $\kappa >0$, the function $\psi (t,x) := - \kappa \phi (x+c_1 t) <0$ satisfies:
\begin{eqnarray*}
&& \displaystyle \partial_t \psi - \partial_{x} (a \partial_x \psi) - c_1 \partial_x \psi - h(t,x) \psi \vspace{3pt}\\
& \leq & \displaystyle \partial_t \psi - \partial_{x}(a \partial_x \psi ) - c_1 \partial_x \psi - g(x+c_1 t) \psi \vspace{3pt}\\
& \leq & \displaystyle \kappa \mu_g \phi (c+c_1 t) \vspace{3pt}\\
& \leq & 0,
\end{eqnarray*}
hence it is a subsolution of the equation satisfied by $\eta$. Furthermore, as $\phi$ is continuous, positive and periodic, and since $\liminf_{t \rightarrow -\infty} \inf_{x \leq x_0} \eta (t,x)  \geq 0$, one can find, for any choice of $\kappa >0$, a time $T <0$ such that
$$\eta (T,x) \geq -\kappa \min \phi, \ \ \forall x \leq x_0.$$
It then follows, from the parabolic maximum principle and the fact that $\eta (t,x_0) \geq 0$ for all $t \leq 0$, that
$$\eta (0,x) \geq -\kappa \max \phi, \ \ \forall x \leq x_0.$$
Since this holds for any $\kappa >0$, we get that $ 0 \leq \eta (0,x) \mbox{ for all } x \leq x_0$
and, since $\eta (0,x_0) >0$, one has by the strong maximum principle that $ \eta (0,x)>0 \mbox{ for all } x \leq x_0.$

Since $U_1$ is steeper than $U_2$, $\eta (0,\cdot)$ must be nonpositive on the left of any zero. Therefore, for all $x \in \R$, $\eta (0,x) = U_2 (0,x) - U_1 (0,x) >0$.
Moreover, from the parabolic comparison principle, we get that for any $t \geq 0$ and $x \in \R$,
$$U_2 (t,x) \geq U_1 (t,x).$$
This clearly implies that $U_2$ has to be faster than $U_1$, that is $c_2 \geq c_1$, and we have reached a contradiction.\\

Consider now the case $c_1 = c_2 \neq 0$. The function $\eta (t,x) := U_2 (t, x +c_1 t) - U_1 (t,x+c_1 t)$ is now periodic in time and converges to 0 as $x \rightarrow -\infty$ uniformly with respect to $t$. Assume by contradiction that $\eta (t'_0, x'_0) <0$ for some $t'_0 \in \R$ and $x'_0 \in \R$. Without loss of generality and by periodicity, we can of course assume that $t'_0 <0$ and, as explained above, as $U_1$ is steeper than $U_2$ we also have that $x'_0 < x_0$. Then, one can define
$$\kappa^* := \min \{ \kappa > 0 \; | \ \eta (t,x) \geq -\kappa \phi (x+c_1 t) \mbox{ for all } t \leq 0 \mbox{ and } x\leq x_0\} >0 ,$$where $\phi$ is again a positive and $L$-periodic function such that
$$\partial_{x} (a \partial_x \phi )+ g \phi = \mu_g \phi.$$
Indeed, this set is non empty from the boundedness of $\eta$, and the minimum is reached since $\eta (t'_0, x'_0) <0$ with $t'_0<0$ and $x'_0 < x_0$.

Furthermore, it is clear that $\eta (t,x) \geq -\kappa^* \phi (x+c_1 t)$ for all $(t,x) \in (-\infty,0] \times (-\infty,x_0]$, with equality for some $(t_1,x_1) \in (-\infty,0] \times (-\infty,x_0)$. Applying the strong maximum principle, we get that $\eta (t,x) \equiv -\kappa^* \phi (x+c_1 t)$, which contradicts the fact that it goes to 0 as $ x \rightarrow -\infty$ for any $t \leq 0$. We conclude that $\eta \geq 0$. That is, we can assume, up to some time shift and without loss of generality, that~$U_1  \leq U_2 $.

Assume that $c_1 >0$. From the asymptotics of $U_1$, it is clear that for any time shift $\tau$ large enough, we have that $U_1 (\tau,0) > U_2 (0,0)$. It follows that the following time shift is well defined:
$$\tau^* := \sup \left\{ \tau \geq 0 \; | \ U_1 (\tau , \cdot ) \leq U_2 (0,\cdot) \right\} < + \infty.$$
When $c_1 < 0$, one instead have that $U_1 (0,0) > U_2 (\tau,0)$ for large enough $\tau$, so that one may rather define $\tau^*$ as the largest nonnegative $\tau$ such that $U_1 (0,\cdot) \leq U_2 (\tau,\cdot)$. Then the proof proceeds similarly as below so that we omit the details of the case $c_1 <0$.

Let us first note that, by continuity, $U_1 (\tau^* , \cdot ) \leq U_2 (0,\cdot).$
Our aim is to show that we in fact have
\begin{equation}\label{eqn:tauabove}
U_1 (\tau^* , \cdot ) \equiv U_2 (0,\cdot).
\end{equation}
It will then immediately follow that $U_1$ is identically equal to $U_2$ up to the time shift $\tau^*$.

Let us argue by contradiction and assume that there exists $x_1 \in \R$ such that $U_1 (\tau^* , x_1 ) < U_2 (0,x_1)$. One can then easily check that
\begin{equation}\label{eqn:tauabovecontrad1}
U_1 (t+\tau^*, x) < U_2 (t,x) \mbox{ for all } (t,x)\in \R^2.
\end{equation}
This indeed follows from the periodicity in time of $U_1 (t+\tau^*,x+c_1t) - U_2 (t,x+c_1t)$ and the strong maximum principle.

It is also clear, by continuity, that for $\tau > \tau^*$ but close enough to $\tau^*$, one has that $U_1 (\tau, x_1) < U_2 (0,x_1)$. Besides, by construction of $\tau^*$, for any $\tau > \tau^*$, there also exists $x_2 \in \R$ such that $U_1 (\tau,x_2) > U_2 (0,x_2)$. Therefore, for some~$\epsilon >0$ we have that $U_1 (\tau,\cdot)$ and $U_2 (0,\cdot)$ intersect at least once for all $\tau \in (\tau^* , \tau^* +\epsilon)$. Because $U_1$ is steeper than $U_2$, and from the periodicity in time of $U_1 (t+\tau,x+c_1t) - U_2 (t,x+c_1t)$, there can only be one (non degenerate) intersection, for each time $t \in \R$,  of $ U_1 (t+\tau,\cdot)$ and $U_2 (t, \cdot)$. Hence, we can define, for any $\tau \in (\tau^*, \tau^* + \epsilon)$, the real-valued function $x (\tau,t)$, which is periodic in its second variable and is such that $$U_1 (t+\tau,x (\tau,t) +c_1t ) = U_2 (t,x (\tau,t)+c_1t ),$$
that is the only intersection of $U_1 (t+\tau,\cdot)$ and $U_2 (t,\cdot)$.\\

Let us now look at the behavior of $x (\tau,t)$ as $\tau \rightarrow \tau^*$. Consider first the case when there exist two sequences $\tau_j \rightarrow \tau^*$ and $t_j$ such that $x (\tau_j,t_j)$ converges to some $x (\tau^*) \in \R$. From the periodicity of $x(\tau,t)$ in the variable $t$, we can assume up to extraction of a subsequence that $t_j \rightarrow t_\infty \in \R$ as $j \rightarrow +\infty$. It immediately follows by passing to the limit as $j \rightarrow +\infty$ that $$U_1 (t_\infty +\tau^*,x (\tau^*)+c_1 t_\infty)= U_2 (t_\infty,x(\tau^*)+c_1 t_\infty).$$
We have reached a contradiction with (\ref{eqn:tauabovecontrad1}).

Consider now the case when there exist two sequences $\tau_j \rightarrow \tau^*$ and $t_j$ such that $x (\tau_j,t_j) \rightarrow +\infty$ as $j\rightarrow +\infty$. Again, we can assume without loss of generality that $t_j \rightarrow t_\infty \in \R$ as $j \rightarrow +\infty$. As $U_1$ is steeper than $U_2$, we know that $U_1 (t_j+ \tau_j,\cdot)$ is above $U_2 (t_j, \cdot)$ on the left of the point $x(\tau_j,t_j) +c_1 t_j$, which goes to $+\infty$ as $j \rightarrow +\infty$. Thus by passing to the limit as $j \rightarrow +\infty$, we get
$$U_1 (t_\infty+ \tau^* , \cdot) \geq U_2 (t_\infty,\cdot),$$
which again is a clear contradiction with (\ref{eqn:tauabovecontrad1}).

The only remaining case is $x (\tau,t) \rightarrow -\infty$ as $\tau \rightarrow \tau^*$, uniformly with respect to its second variable. We proceed similarly as before. Indeed, choose $0 < \delta' < \delta$ and $x_0 \in \R$  such that for any $\tau \in (\tau^*, \tau^*+\epsilon)$,
$$U_1 (t+\tau,x) \in \left[q (x)-\delta,q (x) \right] \; \mbox{ for all } x \leq x_0 + c_1 t,$$
$$U_1 (t+\tau,x) \in \left[q (x)-\delta,q (x) - \delta' \right] \; \mbox{ for all } x = x_0 + c_1 t.$$
Let also $x_1$ such that for all $t \in \mathbb{R}$,
\begin{equation}\label{eqn:ineq42}
U_2 (t,x) \in \left[q (x)-\delta,q (x) \right] \; \mbox{ for all } x \leq x_1 + c_1 t.
\end{equation}
Up to reducing $\epsilon$, for any $\tau \in (\tau^*, \tau^* + \epsilon)$, one has $x (\tau,t) < \min \left\{ x_0 , x_1 \right \}$ for all~$ t\in \R$. Thus, as $U_1$ is steeper,
$$U_1 (t+\tau,x_0+c_1t) \leq U_2 (t,x_0+c_1 t),$$
and moreover, since $U_2$ lies above $U_1 (\cdot +\tau,\cdot)$ on the right of $x(\tau,t) + c_1 t$, and using the inequality (\ref{eqn:ineq42}) on the left of the same point, one gets that
$$U_2 (t,x) \in \left[q (x)-\delta,q (x) \right] \; \mbox{ for all } x \leq x_0 + c_1 t.$$
Similarly as before, the function $\eta(t,x)=U_2 (t,x+c_1t) - U_1 (t+\tau,x+c_1t)$ satisfies
\begin{equation}
\left\{
\begin{array}{rl}
\displaystyle \partial_t \eta = \partial_{x}( a \partial_x \eta ) + c_1 \partial_x \eta + h(t,x) \eta, & \forall (t,x) \in (-\infty,0) \times (-\infty,x_0), \vspace{5pt}\\
\displaystyle \eta (t,x_0) \geq 0, & \forall t \in (-\infty,0),\vspace{5pt}\\
\displaystyle \eta (t,x) \ \mbox{ is periodic in } \; t  & \mbox{ and } \ \eta (t,x) \rightarrow 0 \ \mbox{ as } \; x \rightarrow -\infty,
\end{array}
\right.
\end{equation}
where $h\leq g$ with $\mu_g \leq 0$. As before, we can find a positive and periodic function~$\phi$ such that, for any $\kappa>0$, the function $-\kappa \phi (x+c_1t)$ is a subsolution of the equation above satisfied by $\eta$. 
Again, it is clear that there exists
$$\kappa^* := \min \{ \kappa > 0 \; | \ \eta (t,x) \geq -\kappa \phi(x+c_1 t)  \mbox{ for all } t \leq 0 \mbox{ and } x\leq x_0\} >0.$$
Therefore, using the limiting conditions of $\eta$ near $-\infty$ and $x_0$, we get that $\eta (t,x) \geq - \kappa^* \phi (x+c_1 t)$, with equality for some $(t_1,x_1) \in (-\infty,0]\times(-\infty,x_0]$. From the strong maximum principle, it follows that $\eta (t,x) \equiv -\kappa^* \phi (x+c_1 t) < 0$ for all $t \leq 0$ and $x\leq x_0$, which is a contradiction.

As all the above cases lead to a contradiction, we conclude that (\ref{eqn:tauabove}) holds, that is, $U_2$ is identically equal to some time shift of $U_1$. This ends the proof of Theorem~\ref{th:lem_speeds}.

\section{Uniqueness properties of terraces}\label{sec:uniqueness}

We can now prove our uniqueness theorems. We already gave a brief proof of Theorem~\ref{th:unique_min} in Section~\ref{sec:intro}. We therefore focus here on Theorems~\ref{th:unique_semi} and~\ref{th:unique}.

\subsection{Semi-minimal terraces}

Let us consider a semi-minimal terrace $$\mathcal{T}= \left( (p_k)_{0 \leq k \leq N},(U_k)_{1 \leq k \leq N} \right)$$
connecting 0 to $p$, such that all the speeds $c_k$ of the pulsating traveling waves $U_k$ are not zero. As the minimal terrace is unique when it has non zero speeds, it is enough to prove that $\mathcal{T}$ is minimal to get the first part of Theorem~\ref{th:unique_min}.\\

From Theorem~\ref{th:exists}, we know that there exists a minimal terrace $\mathcal{T}'=((p'_k),(U'_k))$ connecting 0 to $p$, such that any $p'_k$ and $U'_k$ is steeper than any other entire solution. Our goal is to prove that $\mathcal{T}'$ and $\mathcal{T}$ share the same platforms, so that $\mathcal{T}$ is minimal (and we then even have that the $U_k$ and $U'_k$ are the same up to some time shift as $c_k \neq 0$). Proceed by contradiction and assume this is not the case.

Note first that the family $(p'_k)_k$ is included in the family $(p_k)_k$. Let $i \geq 1$ be the smallest integer such that $p_i \neq p'_i$. Then there clearly exists $j >i$ such that
$$p'_{i-1}= p_{i-1} > p_i > p_j = p'_i .$$
One can now choose a real number $a$ large enough so that for all $x \in \mathbb{R}$:
$$U'_i (0,x) < p_i + H(a-x) (p_{i-1} - p_i) =:r(x).$$
Denote by $R(t,x)$ the solution of \eqref{eqn1} with initial datum $r(x)$. It is clear from our definitions that the minimal terrace connecting $p_i$ to $p_{i-1}$ is reduced to the single traveling wave $U_i$ with speed $c_i \neq 0$. From our Theorem~\ref{th:CV}, we thus know that $R(t,x)$ converges to $U_i$ and spreads with the speed $c_i$. Since $U'_i (t,x) < R (t,x)$ for all $x \in \R$ and $t \geq 0$, it follows that $c_i \geq c'_i$, where $c '_i$ is the speed of $U'_i$.

One can then proceed similarly to get that $c'_i \geq c_j$ where $c_j$ is the speed of $U_j$ the $j$-th front of the terrace $\mathcal{T}$. Hence, $c_j \leq c_i$. Under statement~$(a)$ of Theorem~\ref{th:unique_semi}, we have already reached a contradiction, and we conclude that $\mathcal{T}$ is a minimal terrace. In general, from the definition of a terrace, we only get that $c'_i=c_i = c_j$. Furthermore, as speeds of $\mathcal{T}$ are non zero, we have either $c'_i=c_i = c_j >0$ or $c'_i = c_i = c_j < 0$. Assume that the former holds true. As before, it is then known that $p_{i-1}$ is isolated and stable from below (see the proof of Theorem~\ref{th:min_speeds}). Under statement~$(b)$, one can use Assumption~\ref{assumption-speed1} and Theorem~\ref{th:lem_speeds}, as well as the fact that $U'_i$ is steeper than $U_i$, to get that $U'_i \equiv U_i$ up to some time shift. This is a contradiction as both fronts do not share the same asymptotics, namely 
$$U_i (-\infty,\cdot) = p_i (\cdot) > p'_i (\cdot) = U'_i (-\infty, \cdot).$$
The other case can be dealt with the same way. Having reached a contradiction, we get that $U'_i \equiv U_j$, and we can again conclude that $\mathcal{T}$ is a minimal terrace.\\

We now check that there is no other semi-minimal propagating terrace. Let $\mathcal{T}'' =((p''_k) , (U''_k))$ be any other semi-minimal terrace, and assume first that there is no component with zero speed. The argument above then applies, so that $\mathcal{T}''$, $\mathcal{T'}$ and $\mathcal{T}$ are one and the same terrace (up to time shifts), and thus $\mathcal{T}$ is the unique semi-minimal terrace. 

Let us now proceed by contradiction and assume that $c''_i =0$ for some~$i$. Since $\mathcal{T}$ is minimal, there exists some $j$ such that 
$$p_j \leq p''_i \leq p''_{i-1} \leq p_{j-1}.
$$
Furthermore, there exist two integers $i_1 < i$ and $i_2 \geq i$ such that
$$p''_{i_1}=p_{j-1} \mbox{ and } p''_{i_2} = p_{j}.$$
Then the solutions of \eqref{eqn1} with initial data 
\begin{eqnarray*}
&& p''_{i_1 +1} + H(1-x) (p''_{i_1} -p''_{i_1 +1}), \vspace{3pt} \\
&& p_{j} +  H(-x) (p_{j-1} - p_j ),\vspace{3pt} \\
&& p''_{i_2 } + H(-1-x) (p''_{i_2 -1} -p''_{i_2}),
\end{eqnarray*}
converge respectively, thanks to Theorem~\ref{th:CV}, to $U''_{i_1 +1}$, $U_j$ and $U''_{i_2}$. By the comparison principle, it immediately follows that $c''_{i_2} \leq c_j \leq c''_{i_1 +1}$. Since $\mathcal{T}''$ is a terrace and $c''_i =0$, we have that $c''_{i_1 +1} \leq 0$ and $c''_{i_2} \geq 0$. Thus $c''_i = c_j =0$, which contradicts the fact that $\mathcal{T}$ has only non zero speeds. This ends the proof of Theorem~\ref{th:unique_semi}.

\subsection{Uniqueness of the terrace}

We now prove Theorem~\ref{th:unique} and let $$\mathcal{T}= \left( (p_k)_{0 \leq k \leq N},(U_k)_{1 \leq k \leq N} \right)$$ be some terrace connecting 0 to $p$ with non zero speeds. Again, we first prove that it is minimal, then check that there is no other terrace.\\

The proof is very similar to the previous subsection. From Theorem~\ref{th:exists}, we know that there exists a minimal terrace $\mathcal{T}'=\left( (p'_k)_{0 \leq k \leq N'},(U'_k)_{1 \leq k \leq N'}\right)$ connecting 0 to $p$, such that any $p'_k$ and $U'_k$ is steeper than any other entire solution. Our goal is again to prove that $\mathcal{T}'$ and $\mathcal{T}$ are the same terrace up to time shifts.

As before, note first that the family $(p'_k)_{0 \leq k \leq N'}$ is included in the family $(p_k)_{0 \leq k \leq N}$. As each $U_k$ moves with speed $c_k \neq 0$, it follows that each $p'_k$ (for $1 \leq k \leq N'-1$) is isolated and stable from either above or below. Thanks to Assumption~\ref{assumption-speed2} and Theorem~\ref{th:lem_speeds}, we get that for each integer $1 \leq k \leq N'$, the speed $c'_k$ of $U'_k$ is minimal among the speeds of all traveling waves connecting some $q<p'_{k-1}$ to $p'_{k-1}$, but also maximal among the speeds of all traveling waves connecting $p'_k$ to some $q>p'_k$. In particular, for any $k$, there exist $i \leq j$ such that $c'_k \leq c_i$ and $c'_k \geq c_j$, $i$ and $j$ being chosen such that $U_i$ connects $p_i$ to $p_{i-1}=p'_{k-1}$ and $U_j$ connects $p_j=p'_{k}$ to $p_{j-1}$. Since $c_i \leq c_j$ by definition of a propagating terrace, we get that 
$c'_k =c_i =c_j$. As they cannot be zero, and using the second part of Theorem~\ref{th:lem_speeds}, we conclude that $i=k$ and $U_i \equiv U_j \equiv U'_k$ up to some time shifts. This immediately implies that $\mathcal{T}$ and $\mathcal{T}'$ are the same terrace up to time shifts.\\

We now check that there is no propagating terrace with a zero speed component. Proceed by contradiction and assume that there exists some $\mathcal{T}'' =((p''_k) , (U''_k))$ such that $c''_i =0$ for some~$i$. As above, since we now know $\mathcal{T}$ to be minimal, there exist some integers $j$, $i_1$ and $i_2$ such that
$$p''_{i_1}=p_{j-1} \geq p''_{i-1} \geq p''_i \geq p_j= p''_{i_2}.$$
Moreover, as we have proved that $\mathcal{T}$ is also the unique minimal terrace, $U_j$ is steeper than $U''_{i_1 +1}$ and $U''_{i_2}$. It also has non zero speed, and we can again use Assumption~\ref{assumption-speed2} to get that $U_j \equiv U''_{i_1 +1} \equiv U''_{i_2}$ up to time shifts. Thus, $i=i_2 = i_1 +1$ and $c_j =0$, which clearly contradicts our assumptions on $\mathcal{T}$. We conclude that any terrace has no zero speed and, by the argument above, is equal to the unique minimal terrace up to time shifts. Theorem~\ref{th:unique} is now proved.

\section{Discussion on the case with zero speeds}\label{sec:zero}

We now discuss the case when zero speeds occur, which was avoided in the statement of our uniqueness Theorems~\ref{th:unique_min}, \ref{th:unique_semi} and~\ref{th:unique}. Indeed, terraces can obviously no longer be unique up to time shifts anymore if they involve stationary waves. However, one could still expect the platforms of terraces to be unique under some appropriate assumption similar to Assumption~\ref{assumption-speed2}. We will see that this may not be true either, even when all platforms are assumed to be linearly stable. We will end this section with a more precise statement, along with a sketch of its proof.\\

Let us first look at the homogeneous case, that is when the equation~\eqref{eqn1} is invariant by space translation, and assume for simplicity that all platforms of any terrace is locally stable from both below and above. It was already proved in \cite{FifMcL77}, using ODE-inspired proofs, that there exists a unique terrace up to some space shifts, no matter if zero speeds occur or not. In particular, platforms of terraces are also unique.

Without giving the details, we point out that our proof of Theorem~\ref{th:lem_speeds} can easily be extended, in the homogeneous framework, to stationary waves provided that we replace all time shifts by space shifts. In particular, we could use the same argument as in Section~\ref{sec:uniqueness} to reach the same conclusion as in~\cite{FifMcL77}, at least under a stronger assumption similar to Assumption~\ref{assumption-speed2}.

More generally, the key idea is that if there exists a total foliation of stationary waves between any two platforms (which is given, in the homogeneous framework, by all space shifts of a single stationary wave), then all terraces share the same platforms. If not, one could find two distinct stationary waves intersecting each other arbitrarily close to a common upper stable state, and then apply the same method as in the second part of our proof of Theorem~\ref{th:lem_speeds} to reach a contradiction.\\

However, in the heterogeneous case, such a foliation does not always exist. What we know is the following:
\begin{theorem}\label{critical}
Let $q_1 < q_0$ be two stationary solutions of \eqref{eqn1-per}.
For any given $x_0 \in \R$ and $\alpha \in (q_1 (x_0), q_0 (x_0))$, there exists a monotonic in time entire solution $U_{x_0,\alpha}$ of \eqref{eqn1} such that $U_{x_0,\alpha}(0,x_0)=\alpha$, which is steeper than any other entire solution lying between~$q_0$ and~$q_1$.
\end{theorem}

This theorem was proved in~\cite{Nadin12}, where those steepest entire solutions where denoted as ``critical waves". The proof is similar to the argument of Theorem~\ref{th:exists} in Section~\ref{sec:existence} except that, instead of looking at the large time behavior of one Heaviside type initial datum, the Heaviside type initial datum is shifted in order to pin the solution at the value $\alpha$ at the point $x_0$. 

Let us now assume that, for a given $x_0$ and $\alpha$, the critical wave $U_{x_0,\alpha}$ is not a spatially periodic solution of~\eqref{eqn1}. It then connects, as $x \to \pm \infty$, two spatially periodic stationary solutions, which we still denote by $q_1$ and $q_0$ without loss of generality. Two situations then occur. When the average speed of the level sets of this critical wave is not zero, then it is a pulsating traveling wave in the usual sense and, as explained before, it is the unique critical wave, up to time shifts, lying between $q_1$ and $q_0$. When the average speed of the level sets is zero and the equation is homogeneous, then it is a stationary wave, which is the unique critical wave up to space shifts. However, when the average speed of the critical wave is zero and the equation is heterogeneous, then the set of critical waves is ordered (by a straightforward steepness argument) but non trivial and may contain both stationary waves and monotonic in time entire solutions. More precisely, there may exist two critical stationary waves $U_{x,\alpha_1} < U_{x,\alpha_2}$ such that, for any $\alpha_1 < \alpha < \alpha_2$, the critical wave $U_{x,\alpha}$ is a monotonic in time entire solution which converges as $t \to \pm \infty$ to $U_{x,\alpha_1}$ and $U_{x,\alpha_2}$ (in one way or the other depending on the monotonicity).

The occurring of this new object in the dynamics leads to more complicated situations where non minimal terraces exist, even though all platforms are linearly stable. Indeed, let $q_0 > q_1 >q_2$ be three linearly stable states, such that there exists a two platforms semi-minimal terrace with zero speeds. We also assume that for some $\alpha_1 \in (q_1 (0),q_0 (0))$ and $\alpha_2 \in (q_2 (0), q_1 (0))$, the critical waves $U_{0,\alpha_1}$ and $U_{0,\alpha_2}$ (taken steeper than any entire solution lying between respectively $q_1$ and $q_0$, and $q_2$ and $q_1$) are respectively increasing and decreasing in time entire solutions. Then, we know by the arguments developed in this paper that the solution~$u$ associated with Heaviside type initial datum $q_2 (x) + H(-x) (q_0 - q_2)(x)$ converges locally uniformly to a stationary solution~$U$, which is steeper than any other entire solution. We now claim that it is a stationary wave connecting $q_2$ to $q_0$. If it does not, then either $U \geq q_1$ or $U \leq q_1$ (as before, $U$ should be a part of a minimal terrace, whose platforms are included in the decomposition $(q_0, q_1, q_2)$). Assume the former occurs. As in Section~\ref{sec:existence}, one can define for any large integer $n$ the smallest time~$\tau_n$ such that $u$ reaches the value $ \alpha_2$ at the point $nL$ and time $\tau_n$, and then check by a steepness argument that $u$ converges locally uniformly around $(\tau_n, nL)$ to the decreasing in time function $U_{0,\alpha_2}$. However, by definition of $\tau_n$, one has that $\partial_t u (\tau_n,nL) \geq 0$, a contradiction. A similar argument can be performed when $U \leq q_1$, and as announced, $U$ is a stationary wave connecting $q_2$ to $q_0$. We conclude that the minimal terrace only has one platform, while our initial terrace had two.\\

In our last example above, the speed of the upper critical wave was larger than the speed of the lower critical wave, in some sense to be made more rigourous below. Therefore, our initial terrace was not appropriate to describe propagation dynamics, although it matched our definition of a propagating terrace. This means that, whenever stationary waves appear, we need to distinguish between different situations.

Let $q_1 < q_0$ be two periodic stationary solutions that are connected by some pulsating wave, and denote by~$U_{x_0,\alpha}$ the critical waves betwen $q_1$ and $q_0$ as defined by Theorem~\ref{critical}. We will distinguish the following cases:
\begin{enumerate}[$(i)$]
\item There exist $x_0$ and $\alpha \in (q_1 (x_0),q_0 (x_0))$ such that the solution $U_{x_0,\alpha}$ is a pulsating traveling wave with speed $c \neq 0$ (then it is does not depend on $x_0$ and $\alpha$ up to time shifts). We say that $q_1$ and $q_0$ are connected with critical speed $c$.
\end{enumerate}
Otherwise, there exists a stationary wave connecting $q_1$ to $q_0$, and one of the following holds:
\begin{enumerate}[$(i)$]
\setcounter{enumi}1
\item For any $x_0$ and $\alpha \in (q_1 (x_0),q_0 (x_0))$, the solution $U_{x_0,\alpha}$ is a stationary wave. We say that $q_1$ and $q_0$ are connected with critical speed $0$.
\item There exist $x_0$ and $\alpha \in (q_1 (x_0),q_0 (x_0))$ such that the solution $U_{x_0,\alpha}$ is monotonically increasing in time and, for any $x'$ and $\alpha ' \in (q_1 (x'),q_0 (x'))$, the solution $U_{x',\alpha '}$ is never decreasing in time. We say that $q_1$ and $q_0$ are connected with critical speed $0^+$.
\item There exist $x_0$ and $\alpha \in (q_1 (x_0),q_0 (x_0))$ such that the solution $U_{x_0,\alpha}$ is monotonically decreasing in time and, for any $x'$ and $\alpha ' \in (q_1 (x'),q_0 (x'))$, the solution $U_{x',\alpha '}$ is never increasing in time. We say that $q_1$ and $q_0$ are connected with critical speed $0_-$.
\item There exist $x_0$ and $\alpha \in (q_1 (x_0),q_0 (x_0))$ such that the solution $U_{x_0,\alpha}$ is monotonically increasing in time, and another $x_1$ and $\beta \in (q_1 (x_1),q_0 (x_1))$ such that the solution $U_{x_1,\beta}$ is decreasing in time. We say that $q_1$ and $q_0$ are connected with critical speed $0^+_-$.
\end{enumerate}
We point out Theorem~1.7 in~\cite{DHZ2} for a situation where two steady states are connected with critical speed $0^+_-$. Other cases may be constructed in a similar fashion.

The discussion above leads us to order the zero speeds in the following way: for all $c >0$,
$$- c < 0 < c \ \ \mbox {  and  } -c < 0_- < 0^+_- < 0^+ < c.$$ 
Note that the set of admissible speeds is no longer fully ordered. We then formulate the theorem below:
\begin{theorem}\label{th:zero_speed_thm}
Assume that $q_0 >q_1 > q_2$ are three stationary solutions of \eqref{eqn1-per}, and that there exist $\delta >0$ and $g$ an $L$-periodic function such that $\mu_g \leq 0$ and
$$\partial_u f (x,u) \leq g \; \mbox{ for all } x \in  \R, \ u \in \left[ q_i (x)-\delta, q_i (x)+\delta \right] \mbox{ and } i=0,1,2.$$
\begin{enumerate}[$(i)$]
\item If $q_1 < q_0$ are connected with critical speed $c \neq 0^+_-$, then there does not exist a non minimal terrace connecting $q_1$ to $q_0$.
\item If $q_1< q_0$ and $q_2 < q_1$ are connected with critical speeds respectively $c_1$ and $c_2$ in the sense defined above. Then the minimal terrace has only one platform if and only if~$c_1 > c_2$ or $c_1 =c_2 = 0^+_-$.
\end{enumerate}\end{theorem}
We only give a brief sketch of the proof, which follows the ideas exposed in the discussion above.\\

\begin{proof}
First, one may extend Theorem~\ref{th:lem_speeds}, to get that the speed of a critical wave has to be the slowest in a slightly stronger sense: if a critical wave is stationary or increasing in time, then any other monotonic in time entire solution that it intersects close enough to $q_1$ has to be either identically equal, or increasing in time. 

It follows that, if there exist both a critical stationary wave $U$ and a non minimal terrace $\mathcal{T}$ connecting $q_1$ to $q_0$, then $q_1$ and $q_0$ are connected with speed $c\leq 0^+_-$. Indeed, from our Theorem~\ref{th:lem_speeds}, one can easily check that any other propagating terrace connecting $q_1$ to $q_0$ has zero speeds only (this could also be proved using the convergence to critical waves from Heaviside type initial data). Then, the highest component of the non minimal terrace $\mathcal{T}$ provides some stationary wave $U_1$ connecting some $q'$ to $q_0$. By the extended Theorem~\ref{th:lem_speeds} described above, any critical wave intersecting $U_1$ close enough to $q_0$ has to be decreasing in time, hence $q_1$ and $q_0$ are connected with speed $c \leq 0_-^+$. One can get similarly, by looking close to $q_1$, that $c \geq 0^+_-$, thus $c=0^+_-$. This proves statement $(i)$ of Theorem~\ref{th:zero_speed_thm}.

Assume now that $q_1< q_0$ and $q_2 < q_1$ are connected with critical speeds respectively $c_1$ and $c_2$. On one hand, if there exists a critical traveling wave connecting $q_2$ to $q_0$ with speed $c \neq 0$, then we already know that $c_1 > c > c_2$ by Theorem~\ref{th:lem_speeds}. On the other hand, if there exists some critical stationary wave connecting $q_2$ to $q_0$, then it intersects, arbitrarily close to $q_0$, some critical wave $U_1$ connecting $q_1$ to $q_0$. By the extended Theorem~\ref{th:lem_speeds}, $U_1$ is increasing in time, hence $c_1 \geq 0_-^+$. Similarly, one gets that $c_2 \leq 0_-^+$. Thus, there can be a one platform minimal terrace with zero speed only if $c_ 2 \leq 0_-^+ \leq c_1$. Conversely, if $c_1 > c_2$ or $c_1 =c_2 = 0_-^+$, one can proceed as in our previous example above (earlier in this section), looking at the large time behavior from some Heaviside type initial datum, to get that the minimal terrace has only one platform.\end{proof}

\section{A few examples}\label{sec:exemples}

In this work, we have displayed general results in order to describe propagating terraces, and to offer a wide and natural extension of the classical notion of traveling waves. In this section, we illustrate how those results can be applied to several examples, ranging from the standard bistable case to some multistable nonlinearities. 

\subsection{Bistable case}

We first consider $f$ of the bistable type (see Figure~\ref{fig:bistable} for the typical homogeneous case). More precisely, 0 and $p$ are asymptotically stable, respectively from above and below, with respect to \eqref{eqn1-per}, and any other stationary solution $p_1$ of \eqref{eqn1-per} is unstable. Moreover, by unstable we mean that any solution of \eqref{eqn1-per} starting from either below or above $p_1$ diverges from $p_1$ (and thus, it converges to either 0 or $p$ as $t \to +\infty$).

We also make the technical assumption that for any such unstable solution $p_1$, any traveling wave connecting~$p_1$ to $p$ has to be strictly faster than any traveling wave connecting $0$ to $p_1$. This assumption is satisfied, for instance, if any such $p_1$ is linearly unstable.
\begin{figure}[h]
\centering
\includegraphics[width=.6\textwidth]{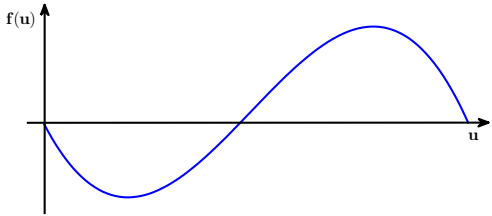}
\caption{Bistable nonlinearity}
\label{fig:bistable}
\end{figure}

We give the following result:

\begin{theorem}\label{th:bistable}
If $f$ is bistable in the sense above, then there exists a traveling wave connecting 0 to $p$.

If moreover Assumption~\ref{assumption-speed2} holds, then there exists a unique admissible speed $c \in \R$ for pulsating traveling waves. When $c \neq 0$, then the profile of the pulsating traveling wave is unique up to time shifts.
\end{theorem}
In~\cite{DGM}, we only proved the existence part under the additional assumption that propagation occurs to $p$ for some compactly supported initial datum. The full existence result was obtained in~\cite{FZ11}, which also dealt with more general discrete and continuous monotone semiflows that we do not consider here. In a similar spirit that consists in looking at the bistable equation as a combination of two monostable parts with opposite speeds, we give an alternative proof which easily follows from our main results. Similar uniqueness results on bistable pulsating waves also recently appeared in the literature; we refer to~\cite{BH12,DHZ}.\\

\begin{proof} Let us first prove that the spatially periodic and bistable (in the sense above) equation admits a pulsating traveling wave. Note first that an intermediate unstable solution of \eqref{eqn1-per} necessarily exists~\cite{Matano84}. Let $p_1$ be such a solution. Then $f$ is of the monostable type both between~$p_1$ and~$p$, and between~0 and~$p_1$, which gives us a decomposition by~\cite{Wein02}, hence a minimal terrace by our Theorem~\ref{th:exists}. Moreover, this terrace is made of a single traveling wave connecting~0 to~$p$. Indeed, if it is not, then it must be made of two pulsating traveling waves connecting respectively~$0$ to~$p_1$ and $p_1$ to $p$. 
However, as all traveling waves between~0 and $p_1$ have strictly slower speed than traveling waves between $p_1$ and~$p$ (this is by our bistable assumption; we also point out that the same hypothesis was made in~\cite{FZ11}), this contradicts our definition of a terrace, namely the fact that speeds must be ordered.

Furthermore, assume now that Assumption~\ref{assumption-speed2} holds. This in particular includes the case when 0 and $p$ are linearly stable, or when $f$ is nonincreasing in some neighborhoods of both 0 and $p$. Then thanks to Theorem~\ref{th:unique} and unless it has zero speed, the traveling wave we just constructed is also the only one connecting 0 to $p$. In other words, the bistable pulsating traveling wave is unique. When the traveling wave is stationary and repeating the same argument, we may only infer that all terraces are made of a single stationary traveling wave: this means that there exists a unique admissible speed for bistable pulsating traveling waves, although there may be different profiles. \end{proof}\\

Let us point out that the existence part of the above theorem is in fact a particular case of Proposition~\ref{proposition}.

\subsection{``Monostable + Bistable" and tristable cases}

We now want to look at some two platform cases. We assume that $f$ is of the bistable type between $p_1$ and $p$ for some positive and periodic stationary solution $p_1 < p$, and either bistable or monostable between 0 and $p_1$, as shown in Figure~\ref{fig:tristable}. We also assume that Assumption~\ref{assumption-speed1} holds.

\begin{figure}[h]
\centering
\includegraphics[width=0.9\textwidth]{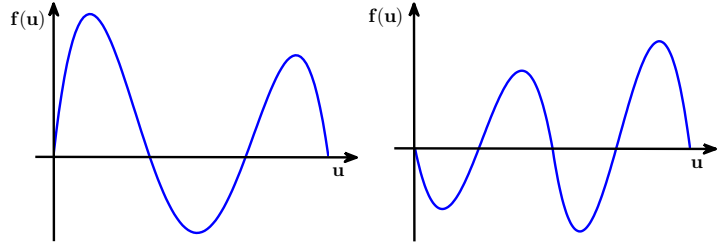}
\caption{(left) Monostable + bistable nonlinearity;\\ (right) Tristable nonlinearity}
\label{fig:tristable}
\end{figure}
Denote by $U_1$ and $U_2$ the unique pulsating waves with minimal speeds $c_1 \neq 0$ and $c_2 \neq 0$ of, respectively, the upper and lower parts. As before, existence of a decomposition immediately gives us existence of a minimal terrace, whose platforms are included in the set $\{ 0, p_1, p \}$. One can check that the minimal terrace has only one platform if and only if $c_1 > c_2$. Indeed,  from Theorem~\ref{th:lem_speeds}, the fronts $U_1$ and $U_2$ are also the steepest traveling waves on their respective intervals. Hence, if $c_1 \leq c_2$, they form a semi-minimal terrace which is also the unique minimal terrace by Theorem~\ref{th:unique_semi}. On the other hand, we have just shown, in the bistable case, that there exists no traveling wave connecting $p_1$ to $p$ with some speed $c \neq c_1$. Therefore, if $c_1 > c_2$, there cannot exist a two platforms minimal terrace. In this case, the minimal terrace has only one platform and its single wave moves with some speed $c \in (c_2, c_1 ) \subset \R^+$ (from either Theorem~\ref{th:lem_speeds} or~\ref{th:CV}). It is also the unique semi-minimal terrace by Theorem~\ref{th:unique_semi}.

Moreover, if the lower part is bistable, it follows from Theorem~\ref{th:unique} that there is no non minimal terrace. On the other hand, if the lower part is monostable, then there is a continuum of admissible speeds $[c_2,+\infty)$ for traveling waves connecting 0 to $p_1$ and, in particular, there always exist non minimal and two platforms terraces connecting 0 to $p$. Note that in this particular case, choosing initial data with slower decay as $x \rightarrow +\infty$, one may construct solutions of the Cauchy problem converging to non minimal terraces as $t \to +\infty$.


\subsection{Quadristable case}

We conclude with a three platforms example. Let us assume that $f$ is quadristable, that is there exist exactly four stable equilibria $p> p_1 > p_2>0$, and all of them are linearly stable, such as in Figure~\ref{fig:quadristable}.

\begin{figure}[h]
\centering
\includegraphics[width=0.9\textwidth]{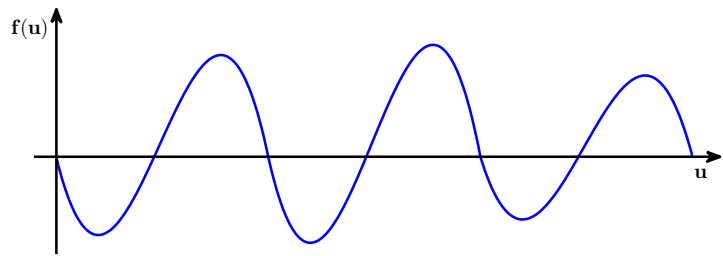}
\caption{Quadristable nonlinearity}
\label{fig:quadristable}
\end{figure}
From the bistable case, we know that there exist three traveling waves $U_1$, $U_2$ and $U_3$ connecting respectively $p_1$ to $p$, $p_2$ to $p_1$ and 0 to $p_2$, that we now assume to move with non zero speeds. If $c_1 \leq c_2 \leq c_3$, then it is a terrace, which is unique according to Theorem~\ref{th:unique}. If $c_1 > c_2$, then we have just shown that there exists a traveling wave connecting $p_1$ to $p$ with speed $c$, so that we are back to the two platforms case. If $c_3 > c$, then the unique (minimal) terrace has two platforms, while if $c_3<c$, then the unique terrace only has one platform. The difficulty is that we do not know a priori what the speed $c$ is, only that $c \in (c_2,c_1)$ thanks to Theorem~\ref{th:lem_speeds}. Thus, knowing the speeds $c_1$, $c_2$ and $c_3$ is not enough to conclude in general. The case $c_3 < c_2$ can be treated similarly.\\

The argument described above gives the sketch of how, by reiterating, one can consider more and more complex nonlinearities (even more degenerate ones, for instance of the ignition type) and extract terraces from any decomposition using our main theorems.


\end{document}